\documentclass[12pt]{article}
\usepackage{amsthm}
\usepackage{amsmath}
\usepackage{amssymb}
\usepackage{latexsym}
\usepackage{enumitem}
\usepackage{comment}
\usepackage{xcolor}
\usepackage{authblk}
\usepackage{tikz}
\usepackage{graphicx}
\usepackage{float}
\usepackage{booktabs}
\usepackage{subcaption}
\usepackage[margin=0.7in]{geometry}

\newtheorem{definition}{Definition}[section]
\newtheorem{theorem}{Theorem}[section]

\newtheorem{proposition}{Proposition}[section]

\newtheorem{corollary}{Corollary}[section]

\newtheorem{remark}{Remark}[section]
\usepackage{stmaryrd}
\input{symbols}

\usepackage[authoryear,round,sort]{natbib}
\usepackage[colorlinks,linkcolor=blue,anchorcolor=blue,citecolor=red]{hyperref}
\DeclareMathOperator*{\argmax}{arg\,max}
\DeclareMathOperator*{\argmin}{arg\,min}

\newcommand{\norm}[1]{\left\lVert#1\right\rVert}

\title{Improving the convergence of Markov chains via permutations and projections}
\author[1]{Michael C.H. Choi\thanks{Email: mchchoi@nus.edu.sg, corresponding author}}
\author[2]{Max Hird\thanks{Email: max.hird.19@ucl.ac.uk}}
\author[3]{Youjia Wang\thanks{Email: e1124868@u.nus.edu}}
\affil[1]{Department of Statistics and Data Science and Yale-NUS College, National University of Singapore, Singapore}
\affil[2]{Department of Statistical Science, University College London, London}
\affil[3]{Department of Statistics and Data Science, National University of Singapore, Singapore}
\date{\today}

\allowdisplaybreaks
\begin{document}

\maketitle

\begin{abstract}

	This paper aims at improving the convergence to equilibrium of finite ergodic Markov chains via permutations and projections. First, we prove that a specific mixture of permuted Markov chains arises naturally as a projection under the KL divergence or the squared-Frobenius norm. We then compare various mixing properties of the mixture with other competing Markov chain samplers and demonstrate that it enjoys improved convergence. This geometric perspective motivates us to propose samplers based on alternating projections to combine different permutations and to analyze their rate of convergence. We give necessary, and under some additional assumptions also sufficient, conditions for the projection to achieve stationarity in the limit in terms of the trace of the transition matrix.  We proceed to discuss tuning strategies of the projection samplers when these permutations are viewed as parameters. Along the way, we reveal connections between the mixture and a Markov chain Sylvester's equation as well as assignment problems, and highlight how these can be used to understand and improve Markov chain mixing. We provide two examples as illustrations. In the first example, the projection sampler (with a suitable choice of the permutation) improves upon Metropolis-Hastings in a discrete bimodal distribution with a reduced relaxation time from exponential to polynomial in the system size, while in the second example, the mixture of permuted Markov chain yields a mixing time that is logarithmic in system size (with high probability under random permutation), compared to a linear mixing time in the Diaconis-Holmes-Neal sampler. Finally, we provide numerical experiments on \textcolor{black}{simple} statistical physics models to illustrate the improved mixing performance of the proposed projection samplers over standard Metropolis-Hastings. \newline	
\end{abstract}

%\tableofcontents
    
\section{Introduction}

Given an ergodic discrete-time Markov chain with transition matrix $P$ and stationary distribution $\pi$, in this paper we focus on improving the convergence of the Markov chain towards $\pi$ via permutations and projections. In the literature, a wide variety of tools and methods have been developed to improve mixing of finite Markov chains. This includes techniques such as lifting \cite{AST21}, non-reversible Markov chain Monte Carlo (MCMC) \cite{DHM00,RBS16}, to name but a few. More recently, there is growing interests in using permutations as a promising technique to accelerate Markov chains, see for example \cite{CD20,BHP23,D24,D24b}.

This manuscript proposes samplers based on projections to improve mixing over the original $P$. These projection samplers depend on some underlying permutation matrices that can be understood as tuning parameters of the algorithms. We summarize several key contributions of the paper as follows.

At the highest level the contributions made in the paper fall into three strands. In the first we identify the projection of a transition matrix $P$ onto the set of $(\pi, Q)$-self-adjoint transition matrices under the Kullback-Leibler divergence, where $Q$ is an isometric involution. The notion of $(\pi, Q)$-self-adjointedness is defined in \cite{A21} where it is shown to hold of various state-of-the-art samplers. We analyze the properties of such a projection, and show that it compares favourably to related transition matrices (such as the original $P$) in relevant quantities such as the spectral gap and asymptotic variance. Having defined the action of a single projection, in the second strand we analyze cycling sequences of these projections. We prove that a limit exists for such sequences, and provide a rate of convergence to it. In the third strand we examine the performance of the projections of existing samplers on a range of target distributions. Here we give both theoretical results (such as improvements to the relaxation time of a Metropolis-Hastings algorithm from exponential to polynomial) and empirical evidence in support of the \textcolor{black}{improved performance of the projections when tested on the one-dimensional Ising model, the Edwards-Anderson spin glass, and the one-dimensional Blume-Capel model}.

To state our contributions in greater detail, and in the order they appear in the text: in Section \ref{sec:basicdef} we use a Pythagorean identity to identify the projection of the transition matrix $P$ onto the set of $(\pi, Q)$-self-adjoint transition matrices, where $Q$ is an isometric involution. The projection takes the following form:
\begin{align}\label{eq:intro}
	\dfrac{1}{2}(P + QP^*Q).
\end{align}
In Section \ref{sec:compare} we show that the projection compares favourably in terms of its Dobrushin coefficient, spectral gap, asymptotic variance, and average hitting time with the following related matrices: $P$, $QP$, $PQ$, and $\alpha P + (1-\alpha)QPQ$ for $\alpha \in [0,1]$.

In Section \ref{sec:altproj} we use a sequence $Q_0,\ldots,Q_{m-1}$ for $m\in\mathbb{N}$ of isometric involutions to generate a sequence of projections. Defining $R_n$ as the result of projecting $P$ onto the $(\pi, Q_0)$-self-adjoint matrices, then the $(\pi,Q_1)$-self-adjoint matrices etc. until we finally we project onto the $(\pi, Q_{n\textup{ mod }m})$-self-adjoint matrices. We prove that a limit $R_\infty$ exists, and we give the rate of convergence to it. In Section \ref{sec:maxspeedlimit} we give necessary and, under additional assumptions, sufficient conditions for when $R_n=\Pi$ for a finite $n\in\mathbb{N}$, where $\Pi$ is just a matrix whose rows are equal to $\pi$. These assumptions involve the trace and the spectrum of $P$.

\textcolor{black}{We proceed to investigate the practical performance of samplers based on \eqref{eq:intro}.} In Section \ref{sec:tuneQ} we discuss tuning strategies of $Q$ to optimize the projection. Interestingly, one strategy lies in finding the $Q$ which solves a Markov chain assignment problem. We discuss adapting $Q$ over the course of a Markov chain, with a technique related to the equi-energy sampler \cite{KZW06}. In Section \ref{sec:MH} we apply the theory we develop to the case that $P$ is a Metropolis-Hastings (MH) sampler, providing a case study in which projecting reduces the energy barriers in the landscape defined by the target, resulting in a decrease in relaxation time from exponential to polynomial in the system size. In Section \ref{sec:general} we assume that the target distribution is uniform, such that $Q$ may be any permutation. We show that when $Q$ is drawn uniformly at random from the set of permutations, the total variation mixing time of the mixture of the permuted chain is, with high probability, at most logarithmic in the system size, while the competing sampler of Diaconis-Holmes-Neal has a linear mixing time. Finally Section \ref{sec:numerical} \textcolor{black}{serves as a proof-of-concept in which we numerically verify the improvements in mixing of projection samplers. Specifically, we} compare variants of the projected kernel in which we variously fix and adapt the projection against a standard MH sampler on the \textcolor{black}{one-dimensional} Ising model, the Edwards-Anderson spin glass, and the \textcolor{black}{one-dimensional} Blume-Capel model.

Some of the subsequent results apply standard techniques in linear algebra, Markov chain theory and information theory. The proofs of these are collected in Section \ref{sec:proof}.

\textcolor{black}{The Appendix in Section \ref{sec:appendix} gives a detailed experimental protocol, as well as a table of commonly used notations.}

\subsection{Related Work}

Perturbing an existing Markov chain kernel is a common way to attempt to improve its properties. \cite{DHM00} analyse the convergence in $\chi^2$ and total variation of a particular `lifted' Markov chain in which the state space is effectively expanded to include a velocity component. Its resulting stationary distribution is uniform $\mathbb{Z}/2n\mathbb{Z}$ for some $n\in\mathbb{N}$. It converges faster than a simple random walk because the velocity component persists through time, reducing the amount of diffusivity. \cite{RBS16} examine chains constructed using reversible and non-reversible perturbations made at the level of the generator of a Markov process, showing that they dominate the unperturbed chain in terms of the spectral gap, asymptotic variance, and large deviation functionals. In the course of doing so they offer concrete instantiations of the perturbed chains they describe in abstract. In general \cite{AST21} captures various perturbed chains (such as non-backtracking chains \cite{kempton2016} and annealed chains \cite{kirkpatrick1983}) under the framework of `lifted' Markov chains, in which one invents dynamics on an extended state space in the hopes of improving convergence properties of the original (now marginal) chain. They show that the mixing times of these `lifted' chains depend sensitively on the adoption of certain design constraints, such as the ability to choose an initial distribution, and whether marginal $\pi$-invariance is preserved upon the lifting operation.

The particular perturbation we make is to project the Markov kernel onto the space of $(\pi, Q)$-self adjoint kernels, where $Q$ is an equi-probability permutation matrix. The resulting projection is a mixture \eqref{eq:intro}. \cite{BHP23} also examine the effects of using permutations on convergences: specifically they identify a cutoff for chains of the form $PQ$ where $P$ is doubly stochastic and $Q$ is a uniformly sampled permutation (hence the stationary distribution is uniform). When $Q$ is predetermined, they improve on a mixing time result of \cite{CD20}. Perturbing with mixtures and permutations is examined in 
\cite{D24, D24b} which demonstrates a cutoff for a Markov kernel of the form
\[
\tilde{P}(x,y)=p_1(x,\sigma(x))P_1(x,y) + p_2(x,\sigma(x))P_2(\sigma(x),\sigma(y))
\]
for all $x,y \in \mathcal{X}$ where $p_1(x,y),p_2(x,y) \in [0,1]$, $\sigma$ is a fixed permutation, and $P_1$ and $P_2$ are transition matrices.

The space we project onto is that of the $(\pi, Q)$-reversible kernels (which we define in Section \ref{sec:basicdef}), whose members are a type of non-reversible Markov kernel. Various characterizations of non-reversibility exist, such as Yaglom reversibility \cite{yaglom1949}, skew detailed balance \cite{turitsyn2011}, and modified detailed balance \cite{fang2014}. Markov kernels that exhibit these types of non-reversibility fall under the following inclusions:
\[
\textup{Yaglom reversibility}=\textup{skew detailed balance} \subseteq(\pi,Q)\textup{-reversibility}\subseteq \textup{modified detailed balance}
\]

In the $(\pi,Q)$-reversible case, \cite{A21} assert general comparison results between asymptotic variances of $(\pi,Q)$-reversible samplers. The following are a selection of $(\pi,Q)$-reversible samplers: Guided Random Walk Metropolis \cite{gustafson1998}, Extra chance Hamiltonian Monte Carlo \cite{campos2015}, and all `lifted' algorithms that satisfy skew detailed balance, such as those in \cite{turitsyn2011} and \cite{sakai2016}. In general a perturbation into $(\pi, Q)$-reversibility is made in order to eliminate the diffusion-like dynamics of existing Markov kernels.

\section{Two types of deformed information divergences and the induced information projections onto the set of $(\pi,Q)$-self-adjoint transition matrices} \label{sec:basicdef}

Let $\mathcal{X}$ be a finite state space and we denote by $\mathcal{L} = \mathcal{L}(\mathcal{X})$ to be the set of transition matrices on $\mathcal{X}$. Analogously we write $\mathcal{P}(\mathcal{X})$ to be the set of probability masses with full support on $\mathcal{X}$, that is, $\min_x \pi(x) > 0$ for $\pi \in \mathcal{P}(\mathcal{X})$. For $m,n \in \mathbb{Z}$ with $m \leq n$, we write $\llbracket m,n \rrbracket := \{m,m+1,\ldots,n-1,n\}$. In particular, when $m = 1$ we write $\llbracket n \rrbracket := \llbracket 1,n \rrbracket$.

Let $\ell^2(\pi)$ be the Hilbert space weighted by $\pi$ endowed with the inner product, for $f,g : \mathcal{X} \to \mathbb{R}$,
\begin{align*}
	\langle f,g \rangle_\pi := \sum_{x \in \mathcal{X}} f(x)g(x) \pi(x).
\end{align*}
The $\ell^2(\pi)$-norm of $f$ is defined to be $\norm{f}_{\pi}^2 = \langle f,f \rangle_\pi$. We also define $\ell^2_0(\pi) := \{f \in \ell^2(\pi);~ \pi(f) = 0\}$.

Given a probability mass $\pi \in \mathcal{P}(\mathcal{X})$, we write $\mathcal{S}(\pi) \subseteq \mathcal{L}$ to be the set of $\pi$-stationary transition matrices, that is, $P \in \mathcal{S}(\pi)$ satisfies $\pi P = \pi$. We also denote by $\mathcal{L}(\pi) \subseteq \mathcal{L}$ to be the set of $\pi$-reversible transition matrices, that is, $P \in \mathcal{L}(\pi)$ satisfies the detailed balance condition with $\pi(x) P(x,y) = \pi(y) P(y,x)$ for all $x,y \in \mathcal{X}$. For $P \in \mathcal{S}(\pi)$, we write $P^* \in \mathcal{S}(\pi)$ to be the time-reversal or the $\ell^2(\pi)$-adjoint of $P$. Thus, $P \in \mathcal{L}(\pi)$ if and only if $P = P^*$.

Let $Q: \ell^2(\pi) \to \ell^2(\pi)$ be an isometric involution on $\mathcal{X}$ with respect to $\pi$ as in \cite{A21}, that is, $Q$ satisfies $Q^2 = I$ and $Q^* = Q$. We write $\mathcal{I}(\pi) = \mathcal{I}(\pi, \mathcal{X})$ to be the set of isometric involution matrices on $\mathcal{X}$ with respect to $\pi$. $L \in \mathcal{S}(\pi)$ is said to be $(\pi,Q)$-self-adjoint if and only if $L^* = QLQ$. This is also equivalent to say that $QL$ is $\ell^2(\pi)$-self-adjoint, and when $Q$ is also a Markov kernel, $QL \in \mathcal{L}(\pi)$. We write $\mathcal{L}(\pi,Q) \subseteq \mathcal{L}$ to be the set of $(\pi,Q)$-self-adjoint transition matrices. In the special case of $Q = I$, we recover that $\mathcal{L}(\pi,I) = \mathcal{L}(\pi)$.

We now characterize $\mathcal{I}(\pi) \cap \mathcal{L}$ in the finite state space setting. Let $\mathbf{P}$ be the set of permutations on $\mathcal{X}$. Let $\psi \in \mathbf{P}$ be a permutation, and $Q_\psi$ be the induced permutation matrix with entries $Q_\psi(x,y) := \delta_{y = \psi(x)}$ for all $x,y \in \mathcal{X}$, where $\delta$ is the Dirac mass function. Define a set of permutations with respect to $\pi$ to be
\begin{align*}
	\Psi(\pi) := \{\psi \in \mathbf{P};~ \forall x \in \mathcal{X},~ \psi(\psi(x)) = x, ~ \pi(x) = \pi(\psi(x))\}.
\end{align*}

\begin{proposition}\label{prop:characterizIpi}
	$$\mathcal{I}(\pi)\cap \mathcal{L} = \{Q_\psi;~ \psi \in \Psi(\pi)\}.$$
\end{proposition}

\begin{proof}
	We first prove that $\{Q_\psi;~ \psi \in \Psi(\pi)\} \subseteq \mathcal{I}(\pi) \cap \mathcal{L}$. We check that $Q^2_\psi(x,y) = \delta_{y = x}$ and hence $Q^2_\psi = I$. The detailed balance condition is also satisfied since $\pi(x) Q_\psi(x,y) = \pi(x) \delta_{y = \psi(x)} = \pi(\psi(x)) \delta_{x = \psi(y)} = \pi(y) Q_\psi(y,x)$. This shows $Q_\psi \in \mathcal{I}(\pi) \cap \mathcal{L}$.
	
	Next, we prove the opposite direction. Precisely, if $Q \in \mathcal{I}(\pi) \cap \mathcal{L} $, then by \cite[Remark 4(a)]{Miclo18} $Q = Q_\sigma$ where $\sigma$ is a permutation such that $\sigma^{-1} = \sigma$. Since $Q_\sigma$ is $\pi$-reversible, we check that $\pi(x) = \pi(x) Q_\sigma(x, \sigma(x)) = \pi(\sigma(x)) Q_\sigma(\sigma(x),x) = \pi(\sigma(x))$. This verifies that $\sigma \in \Psi(\pi)$, which completes the proof.
\end{proof}

Note that since the identity mapping $\psi(x) = x$ belongs to $\Psi(\pi)$ for all $\pi \in \mathcal{P}(\mathcal{X})$, $I = Q_\psi \in \mathcal{I}(\pi) \cap \mathcal{L}$, and hence the set $\mathcal{I}(\pi) \cap \mathcal{L}$ is non-empty. We also note that $\pm(2\Pi - I) \in \mathcal{I}(\pi)$ but these are not transition matrices, where $\Pi$ is the matrix with each row equals to $\pi$.

Another remark is that, for $\psi \in \Psi(\pi)$, this is an ``equi-probability" permutation with respect to $\pi$ since we require $\pi(x) = \pi(\psi(x))$ for all $x$. This connection with the equi-energy sampler \cite{KZW06} is further highlighted in Section \ref{sec:tuneQ}.

As another important point to note, $QP$ or $PQ$ have been proposed and analyzed in the literature as promising samplers over the original $P$, see for example \cite{BHP23,CD20} and the references therein. In the special case of $P \in \mathcal{L}(\pi)$ and $Q \in \mathcal{I}(\pi) \cap \mathcal{L}$, we see that
\begin{align*}
	(QP)^* = Q(QP)Q, \quad (PQ)^* = Q(PQ)Q,
\end{align*}
and hence both $QP, PQ \in \mathcal{L}(\pi,Q)$. That is, they are $(\pi,Q)$-self-adjoint transition matrices, even if they are non-reversible with respect to $\pi$.

We now introduce two types of deformed Kullback-Leibler (KL) divergences that depend on $Q$.

\begin{definition}[$Q$-left-deformed and $Q$-right-deformed KL divergences]
	Let $\pi \in \mathcal{P}(\mathcal{X})$. Let $Q \in \mathcal{I}(\pi) \cap \mathcal{L}$ be an isometric involution transition matrix, $P,L \in \mathcal{L}$. The $Q$-left-deformed KL divergence from $L$ to $P$ with respect to $\pi$ is defined to be
	\begin{align*}
		{}^Q D^{}_{KL}(P \| L) := D^{\pi}_{KL}(QP \| QL) := \sum_x \pi(x) \sum_y QP(x,y) \ln \left(\dfrac{QP(x,y)}{QL(x,y)}\right),
	\end{align*}
	where the usual conventions of $0 \ln (0/0) := 0$ and $0 \cdot \infty := 0$ applies. Note that the dependency on $\pi$ of ${}^QD_{KL}$ is suppressed.
	
	Similarly, we define the $Q$-right-deformed KL divergence from $L$ to $P$ with respect to $\pi$ to be
	\begin{align*}
		D^{Q}_{KL}(P \| L) := D^{\pi}_{KL}(PQ \| LQ).
	\end{align*}
\end{definition}

Note that in the special case of $Q = I$, ${}^I D^{}_{KL} = D^{I}_{KL}$ is the classical KL divergence rate from $L$ to $P$ when $P \in \mathcal{S}(\pi)$.

In the next proposition, we summarize a few properties of $D_{KL}^Q$ and ${}^Q D_{KL}$:

\begin{proposition}\label{prop:deformKLprop}
	Let $\pi \in \mathcal{P}(\mathcal{X})$ and $Q \in \mathcal{I}(\pi) \cap \mathcal{L}$ be an isometric involution transition matrix. For $P,L \in \mathcal{L}$, we have the following:
	\begin{enumerate}
		\item(Non-negativity) 
		\begin{align*}
			{}^Q D_{KL}(P \| L) \geq 0.
		\end{align*}
		Equality holds if and only if $QP = QL$ if and only if $P = L$. Similarly, 
		\begin{align*}
			D_{KL}^Q(P \| L) \geq 0.
		\end{align*}
		Equality holds if and only if $PQ = LQ$ if and only if $P = L$.
		
		\item(Duality) Let $P,L \in \mathcal{S}(\pi)$.
		\begin{align*}
			{}^Q D_{KL}(P \| L) = D_{KL}^Q(P^* \| L^*).
		\end{align*}
	\end{enumerate}
\end{proposition}

For $P \in \mathcal{S}(\pi)$ and $Q \in \mathcal{I}(\pi) \cap \mathcal{L}$ being an isometric involution transition matrix, we define
\begin{align}\label{def:overlineP}
	\overline{P} = \overline{P}(Q) := \dfrac{1}{2}(P + QP^*Q).
\end{align}
It can readily be seen that $Q \overline{P} Q = \overline{P}^*$, and hence $\overline{P} \in \mathcal{L}(\pi,Q)$. In the special case of $Q = I$, we recover that $\overline{P}(I)$ is the additive reversiblization of $P$. We also note that $\overline{P}(Q)$ can be interpreted as a specific mixture of permuted Markov chains in the sense of \cite{D24b}.

The next result presents a Pythagorean identity, which can be interpreted as the property that $\overline{P}$ is the closest $(\pi,Q)$-self-adjoint transition matrix to a given $P$:

\begin{proposition}\label{prop:basic}
	Let $P \in \mathcal{S}(\pi)$ and $Q \in \mathcal{I}(\pi) \cap \mathcal{L}$ be an isometric involution transition matrix. For $L \in \mathcal{L}(\pi,Q)$, we then have
	\begin{align}
		{}^Q D_{KL}(P \| L) &= {}^Q D_{KL}(P \| \overline{P}) + {}^Q D_{KL}(\overline{P} \| L), \label{eq:QPyth}\\
		D_{KL}^Q(P \| L) &= D_{KL}^Q(P \| \overline{P}) +  D_{KL}^Q(\overline{P} \| L). \label{eq:PythQ}
	\end{align}
\end{proposition}
\begin{remark}[Generalization to $f$-divergences]
	The focus of this paper is on convergence to equilibrium of finite Markov chains under either the KL divergence or the squared-Frobenius norm. In this remark, we discuss generalizing the results thus far to $f$-divergences. Let $f: \mathbb{R}_+ \to \mathbb{R}$ be a convex function that satisfies $f(1) = 0$. Let $\pi \in \mathcal{P}(\mathcal{X})$, $Q \in \mathcal{I}(\pi) \cap \mathcal{L}$ be an isometric involution transition matrix, $P,L \in \mathcal{L}$. The $\pi$-weighted $f$-divergence from $L$ to $P$ is defined to be
	\begin{align*}
		D^{\pi}_f(P \| L) := \sum_{x,y} \pi(x) L(x,y) f\left(\dfrac{P(x,y)}{L(x,y)}\right),
	\end{align*}
	where several standard conventions apply in the right hand side. Analogously, we define the $Q$-left-deformed and $Q$-right-deformed $f$-divergence to be respectively
	\begin{align*}
		{}^Q D^{}_{f}(P \| L) := D^{\pi}_{f}(QP \| QL), \quad D^{Q}_{f}(P \| L) := D^{\pi}_{f}(PQ \| LQ). 
	\end{align*}
	In the special case of $f(t) = t \ln t$, we recover the KL divergence counterparts.

	Suppose that $f$ is continuous. For a fixed $P$, since $L \mapsto {}^Q D^{\pi}_{f}(P \| L)$ and $L \mapsto D^{Q}_{f}(P \| L)$ are convex in $L$ (see e.g. \cite{polyanskiy2022information}) and $\mathcal{L}(\pi,Q)$ is a non-empty convex and compact set, a minimizer exists for the following projection problems
	\begin{align*}
		\argmin_{L \in \mathcal{L}(\pi,Q)} {}^Q D^{}_{f}(P \| L), \quad \argmin_{L \in \mathcal{L}(\pi,Q)} D^{Q}_{f}(P \| L).
	\end{align*} 

	To determine a closed-form expression or to establish a Pythagorean identity under common choices of $f$, one can follow similar calculations as in \cite{billera2001geometric,DM09,wolfer2021information,CW23}. As an illustration, for $P \in \mathcal{S}(\pi)$ we let $e(P)$ to be the so-called exponential reversiblization of $P$. Using the Pythagorean identity developed in \cite[Theorem $6.1$]{wolfer2021information}, we see that, for the choice of $f(t) = - \ln t$ that generates the reverse KL divergence and $L \in \mathcal{L}(\pi,Q)$,
	\begin{align*}
		{}^Q D^{}_{f}(P \| L) = D^{\pi}_{f}(QP \| QL) &= D^{\pi}_f(QP \| e(QP)) + D^{\pi}_f(e(QP) \| QL) \\
		&= {}^Q D^{}_{f}(P \| Qe(QP)) + {}^Q D^{}_{f}(Qe(QP) \| L).
	\end{align*}
	Thus, under the $Q$-left-deformed reverse KL divergence, the unique projection of $P$ onto $\mathcal{L}(\pi,Q)$ is $Qe(QP)$ rather than $\overline{P}(Q)$.
	
	Generalizating the above discussions to $f$-divergences, one can expect the projection to be of the form $Q g(QP)$ or $g(PQ)Q$, where the map $g$ is a type of reversiblization. For example, $g$ is the additive reversiblization under KL divergence while $g = e$ is the exponential reversiblization under the reverse KL divergence. We leave this direction as a future work.
\end{remark}

\subsection{Projection under the squared-Frobenius norm}\label{subsec:projectfrob}

In this subsection, we consider projection of $P$ under the squared-Frobenius norm. Let $n = |\mathcal{X}|$ and we write $\mathcal{M}$ to be the set of real-valued matrices on $\mathcal{X}$, that is,
\begin{align*}
	\mathcal{M} = \mathcal{M}(\mathcal{X}):= \{M \in \mathbb{R}^{n \times n}\},
\end{align*}
equipped with the Frobenius inner product defined to be, for $M,N \in \mathcal{M}$,
\begin{align*}
	\langle M,N \rangle_F := \mathrm{Tr}(M^* N)
\end{align*}
and the induced Frobenius norm $\norm{A}_F := \sqrt{ \langle A,A \rangle_F}$, where $\mathrm{Tr}(M)$ is the trace of $M$ and $M^*(x,y) := \frac{\pi(y)}{\pi(x)} M(y,x)$ for all $x,y$ is the $\ell^2(\pi)$-adjoint of $M$. Define for $Q \in \mathcal{I}(\pi) \cap \mathcal{L}$,
%For $\pi \in \mathcal{P}(\mathcal{X})$, let $M^*$ be the $\ell^2(\pi)$-adjoint of $M$, that is, for all $x,y \in \mathcal{X}$, we have $M^*(x,y) = \frac{\pi(y)}{\pi(x)}M(y,x)$. 
\begin{align*}
	\mathcal{M}(\pi,Q) := \{M \in \mathcal{M};~ M = QMQ\}.
\end{align*}
We also define
\begin{align*}
	\overline{M}(Q) := \dfrac{1}{2}(M + QMQ).
\end{align*}
Note that $\overline{M}(Q)$ is a projection in the functional analytic sense, since it can be checked that for all $M \in \mathcal{M}$,
\begin{align*}
	\overline{\overline{M}(Q)}(Q) = \dfrac{1}{2}(\overline{M}(Q) + Q\overline{M}(Q)Q) = \overline{M}(Q).
\end{align*}
In fact, it is an orthogonal projection. To see that, we observe that
\begin{align*}
	\langle \overline{M}(Q),N \rangle_F = \dfrac{1}{2} \langle M,N \rangle_F + \dfrac{1}{2} \mathrm{Tr}(N^* Q M Q) &= \dfrac{1}{2} \langle M,N \rangle_F + \dfrac{1}{2} \mathrm{Tr}(Q N^* Q M) \\
	&= \langle M,\overline{N}(Q) \rangle_F,
\end{align*}
where the second equality follows from the cyclic property of trace and $Q^* = Q$. Note that $\mathcal{M}(\pi,Q)$ is a subspace of the Hilbert space $(\mathcal{M},\langle \cdot,\cdot \rangle_F)$. However, $\mathcal{L}(\pi,Q)$ is not a subspace since it is not closed under scalar multiplication. For example, if $P \in \mathcal{L}(\pi,Q)$ and $\alpha < 0$, then $\alpha P \notin \mathcal{L}(\pi,Q)$.

We now state that $\overline{M}(Q)$ is the unique orthogonal projection of $M$ onto $\mathcal{M}(\pi,Q)$ under the squared-Frobenius norm:

\begin{proposition}[Pythagorean identity under squared-Frobenius norm]\label{prop:projectFrob}
	Let $M \in \mathcal{M}$, $Q \in \mathcal{I}(\pi) \cap \mathcal{L}$ and $N \in \mathcal{M}(\pi,Q)$. We have
	\begin{align*}
		\norm{M - N}_F^2 = \norm{M - \overline{M}(Q)}_F^2 + \norm{\overline{M}(Q) - N}_F^2.
	\end{align*}
	In particular, this yields $\overline{M}(Q)$ is the unique projection of $M$ onto $\mathcal{M}(\pi,Q)$ under the squared-Frobenius norm.
\end{proposition}

Using both Proposition \ref{prop:basic} and \ref{prop:projectFrob}, we see that, for a given $P \in \mathcal{L}(\pi)$, not only $\overline{P}(Q)$ is the unique information projection of $P$ onto $\mathcal{L}(\pi,Q)$ under the deformed divergences $D^Q_{KL}$ and ${}^Q D_{KL}$, it is also the unique orthogonal projection of $P$ onto $\mathcal{M}(\pi,Q)$ under the squared-Frobenius norm.

\section{Comparisons of some samplers}\label{sec:compare}

Given $\pi \in \mathcal{P}(\mathcal{X}), P \in \mathcal{S}(\pi)$ and $Q \in \mathcal{I}(\pi) \cap \mathcal{L}$ being an isometric involution transition matrix, the aim of this section is to compare the convergence of various natural samplers associated with these matrices, such as $P$, $QP$, $PQ$, $QPQ$, $\overline{P}(Q)$ or more generally the mixture $\alpha P + (1-\alpha) QPQ$ for $\alpha \in [0,1]$.

\subsection{Comparisons of entropic parameters}

In this section, we compare parameters related to the KL divergence and entropy.

To this end, let us recall that the KL-divergence Dobrushin coefficient \cite[Definition $2.7$]{WC23} is defined to be
\begin{definition}\label{def:KLDobrushin}
	Let $\pi \in \mathcal{P}(\mathcal{X}), P \in \mathcal{S}(\pi)$. Then the KL-divergence Dobrushin coefficient of $P$, $c_{KL}(P)$, is defined to be
	\begin{align*}
		c_{KL}(P) := \max_{M,N \in \mathcal{S}(\pi), M \neq N} \dfrac{D^{\pi}_{KL}(MP \| NP)}{D^{\pi}_{KL}(M \| N)} \in [0,1].
	\end{align*}
\end{definition}

It can readily be seen that, for $n \in \mathbb{N}$,
\begin{align*}
	D^\pi_{KL}(P^n \| \Pi) \leq c_{KL}(P)^{n-1} D^\pi_{KL}(P \| \Pi),
\end{align*}
and hence $c_{KL}(P)$ can be understood as an upper bound on the convergence rate of $P^n$ towards $\Pi$ under $D^\pi_{KL}$. Thus, a smaller value of $c_{KL}(P)$ indicates a smaller upper bound. We remark that the classical Dobrushin coefficient is defined by replacing $D^\pi_{KL}$ with the $\pi$-weighted total variation distance. More generally, one can define a $f$-divergence Dobrushin coefficient as in \cite{Raginsky2016}.

Making use of the KL-divergence Dobrushin coefficient, we first show that, for $\pi$-stationary transition matrices, the original $\pi$-weighted KL divergence in fact coincides with the deformed KL divergences that we introduce earlier in Section \ref{sec:basicdef}.

\begin{proposition}\label{prop:deformKLorigKL}
	Let $\pi \in \mathcal{P}(\mathcal{X}), M,N \in \mathcal{L}$ and $Q \in \mathcal{I}(\pi) \cap \mathcal{L}$. We have
	\begin{align*}
		D^\pi_{KL}(M \| N) = D^Q_{KL}(M \| N).
	\end{align*}
	If $M,N \in \mathcal{S}(\pi)$, then
	\begin{align*}
		D^\pi_{KL}(M \| N) = {}^Q D_{KL}(M \| N).
	\end{align*}
\end{proposition}

Our second result states that, when measured by $D^{\pi}_{KL}$, the KL divergence from $\Pi$ to any of $P, PQ, QP, QPQ$ are all the same. Analogous results hold for the KL-divergence Dobrushin coefficient. We also demonstrate that the projection is trace-preserving in the sense that $\mathrm{Tr}(P) = \mathrm{Tr}(\overline{P}(Q))$, a property that we shall utilize in Section \ref{sec:maxspeedlimit} below.

\begin{proposition}\label{prop:converge}
	Let $\pi \in \mathcal{P}(\mathcal{X}), P \in \mathcal{S}(\pi)$ and $Q \in \mathcal{I}(\pi) \cap \mathcal{L}$ to be an isometric involution transition matrix. Let $\Pi$ be the matrix where each row equals to $\pi$. We have
	\begin{itemize}
		\item(One-step contraction measured by $D^{\pi}_{KL}$) \begin{align}\label{eq:compare}
			D^{\pi}_{KL}(P \| \Pi) &= D^{\pi}_{KL}(P Q \| \Pi) = D^{\pi}_{KL}(Q P \| \Pi) = D^{\pi}_{KL}(Q P Q\| \Pi).
		\end{align}
		
		\item(KL-divergence Dobrushin coefficient)
		\begin{align}\label{eq:compare3}
			c_{KL}(P) = c_{KL}(PQ) = c_{KL}(QP) = c_{KL}(QPQ).
		\end{align}
	
		\item(Projection is trace-preserving)
		\begin{align}\label{eq:compare4}
			\mathrm{Tr}(P) = \mathrm{Tr}(QP^*Q) = \mathrm{Tr}(\overline{P}(Q)).
		\end{align}
	\end{itemize}
	
\end{proposition}

Our next result states that, the KL divergence from $\Pi$ to $\overline{P}(Q)$ is at least smaller than that to $P$.

\begin{proposition}[Pythagorean identity]\label{prop:converge2}
	Let $\pi \in \mathcal{P}(\mathcal{X}), P \in \mathcal{S}(\pi)$ and $Q \in \mathcal{I}(\pi) \cap \mathcal{L}$ to be an isometric involution transition matrix. Let $\Pi$ be the matrix where each row equals to $\pi$. We have
	\begin{align*}
		D^{\pi}_{KL}(\overline{P} \| \Pi) &\leq  D^{\pi}_{KL}(P  \| \overline{P} ) +  D^{\pi}_{KL}(\overline{P} \| \Pi) = D^{\pi}_{KL}(P \| \Pi),
	\end{align*}
	and the equality holds if and only if $P \in \mathcal{L}(\pi,Q)$ so that $\overline{P}(Q) = P$. 
	
	Similarly, if $P$ is further assumed to be $\pi$-reversible, then
	\begin{align*}
		c_{KL}(\overline{P}(Q)) &\leq c_{KL}(P).
	\end{align*}
\end{proposition}

\begin{remark}
	Note that by Proposition \ref{prop:converge}, we have
	\begin{align*}
		D^{\pi}_{KL}(\overline{P}(Q) \| \Pi) = D^{\pi}_{KL}((1/2)(PQ + QP^*) \| \Pi).
	\end{align*}
\end{remark}

In view of Proposition \ref{prop:converge} and \ref{prop:converge2}, given an arbitrary $\pi$-stationary $P$, it is thus advantageous to use the transition matrix $\overline{P^n}(Q)$ over other competing samplers such as $P^n, Q P^n, P^n Q, QP^nQ$, when measured by $D^{\pi}_{KL}$.

For $\alpha \in [0,1]$, we define 
\begin{align}\label{eq:PbaralphaQ}
	\overline{P}_{\alpha}(Q) := \alpha P + (1-\alpha) QPQ.
\end{align}
In the special case of $\alpha = 1/2$, we recover $\overline{P}_{1/2}(Q) = \overline{P}(Q)$ when $P \in \mathcal{L}(\pi)$. Also, we compute that
\begin{align}\label{eq:Pbarbar}
	\overline{\overline{P}_\alpha(Q)}(Q) = \overline{\alpha P + (1-\alpha) QPQ}(Q) = \overline{P}(Q).
\end{align}

An interesting consequence of Proposition \ref{prop:converge2} is that the choice of $\alpha = 1/2$ is optimal within the family $(\overline{P}_\alpha(Q))_{\alpha \in [0,1]}$ in the sense that it minimizes the KL divergence $D^{\pi}_{KL}$ and the KL-divergence Dobrushin coefficient when $P$ is $\pi$-reversible:

\begin{corollary}[Optimality of $\alpha = 1/2$]\label{cor:optimal1/2}
	Let $P \in \mathcal{L}(\pi)$ and $Q \in \mathcal{I}(\pi) \cap \mathcal{L}$ be an isometric involution transition matrix. We have
	\begin{align*}
		\min_{\alpha \in [0,1]} D^{\pi}_{KL}(\overline{P}_\alpha(Q) \| \Pi) &= D^{\pi}_{KL}(\overline{P}(Q) \| \Pi), \\
		\min_{\alpha \in [0,1]} c_{KL}(\overline{P}_\alpha(Q)) &= c_{KL}(\overline{P}(Q)).
	\end{align*}
\end{corollary}

The proof can readily be seen from Proposition \ref{prop:converge2} by replacing $P$ therein by $\overline{P}_\alpha(Q)$ and using \eqref{eq:Pbarbar}. It is interesting to note that $P$ and $QPQ$ are ``equivalent" in terms of their one-step contraction and Dobrushin coefficient, but randomly choosing to move according to $P$ or $QPQ$ at each step using a fair coin (i.e. using $(1/2)(P+QPQ)$) improves the performance.

\subsection{Comparisons of spectral parameters}

The aim of this subsection is to compare spectral parameters of various Markov chains. To this end, let us now fix a few notations.

For a matrix $M \in \mathcal{M}$, we write $\lambda(M)$ to be the set of eigenvalues of $M$ counted with multiplicities. For a self-adjoint $M$, we denote by $\lambda_1(M) \geq \lambda_2(M) \geq \ldots \lambda_{|\mathcal{X}|}(M)$ to be its eigenvalues arranged in non-increasing order. The right spectral gap $\gamma(P)$ of $P \in \mathcal{L}(\pi)$ is defined to be $\gamma(P) := 1 - \lambda_2(P)$, while the second largest eigenvalue in modulus $\mathrm{SLEM}(P)$ of $P$ is defined to be 
\begin{align}\label{def:SLEMP}
	\mathrm{SLEM}(P) := \max\{\lambda_2(P), |\lambda_{|\mathcal{X}|}(P)|\}.
\end{align}
We shall be interested in several hitting and mixing time parameters that are related to the spectrum of ergodic $P \in \mathcal{L}(\pi)$. Let $\tau_A = \tau_A(P) := \inf\{n \in \mathbb{N};~ X_n \in A\}$ be the first hitting time of the set $A$ of the Markov chain $(X_n)_{n\in \mathbb{N}}$ associated with $P$, where the usual convention of $\inf \emptyset := \infty$ applies. We write $\tau_x := \tau_{\{x\}}$ for $x \in \mathcal{X}$. The average hitting time $t_{av}(P)$ of $P$ is defined to be
\begin{align*}
	t_{av}(P) := \sum_{x,y} \pi(x)\pi(y) \mathbb{E}_x(\tau_y).
\end{align*}
The eigentime identity \cite{AF02} relates $t_{av}(P)$ to the spectrum of $P$ via
\begin{align*}
	t_{av}(P) = \sum_{i=2}^{|\mathcal{X}|} \dfrac{1}{1-\lambda_i(P)}.
\end{align*}
The relaxation time of $P$ is given by 
$$t_{rel}(P) := \dfrac{1}{\gamma(P)}.$$

Under the assumptions that $P \in \mathcal{L}(\pi)$ and $Q \in \mathcal{I}(\pi) \cap \mathcal{L}$ is an isometric involution transition matrix, $QPQ$ is a similarity transformation of $P$ and hence various spectral parameters between these two coincide.

\begin{proposition}\label{prop:similar}
	Let $P \in \mathcal{L}(\pi)$ and $Q\in \mathcal{I}(\pi) \cap \mathcal{L}$ be an isometric involution transition matrix. We have 
	$$\lambda(QPQ) = \lambda(P).$$
	Consequently, this leads to
	\begin{align*}
		t_{av}(QPQ) = t_{av}(P), \quad t_{rel}(QPQ) = t_{rel}(P).
	\end{align*}
\end{proposition}

In the next results, we compare the eigenvalues of $\alpha P + (1-\alpha)QPQ$ for $\alpha \in (0,1)$ with that of $P$ (or $QPQ$). A natural tool to utilize in this context is the Weyl's inequality.

\begin{proposition}\label{prop:spectralspeedup}
	Let $P \in \mathcal{L}(\pi)$ and $Q \in \mathcal{I}(\pi) \cap \mathcal{L}$ be an isometric involution transition matrix. Fix $\alpha \in (0,1)$ and recall that $\overline{P}_{\alpha}(Q)$ is introduced in \eqref{eq:PbaralphaQ}. Let $n := |\mathcal{X}|$. We have 
	\begin{itemize}
		\item $\lambda_2(\overline{P}_{\alpha}(Q)) \leq \lambda_2(P)$, where the equality holds if and only if there exists a common eigenvector $f$ such that $\overline{P}_{\alpha}(Q)f = \lambda_2(\overline{P}_{\alpha}(Q))f$, $Pf = \lambda_2(P) f$ and $QPQf = \lambda_2(QPQ) f$.
		
		\item $\lambda_n(P) \leq \lambda_n(\overline{P}_{\alpha}(Q))$, where the equality holds if and only if there exists a common eigenvector $g$ such that $\overline{P}_{\alpha}(Q)g = \lambda_n(\overline{P}_{\alpha}(Q))g$, $Pg = \lambda_n(P) g$ and $QPQg = \lambda_n(QPQ) g$.
	\end{itemize}
	Consequently, this leads to
	\begin{align*}
		\mathrm{SLEM}(\overline{P}_\alpha(Q)) \leq \mathrm{SLEM}(P).
	\end{align*}
	If $P$ (and hence $QPQ$) is further assumed to be positive-semi-definite, then
	\begin{align*}
		\max\{\alpha,1-\alpha\} \lambda_2(P) &\leq \lambda_2(\overline{P}_{\alpha}(Q)) \leq \lambda_2(P), \\
		\max\{\alpha,1-\alpha\} \mathrm{SLEM}(P) &\leq \mathrm{SLEM}(\overline{P}_\alpha(Q)) \leq \mathrm{SLEM}(P).
	\end{align*}
\end{proposition}

In view of Proposition \ref{prop:spectralspeedup}, it is advantageous to consider the family of samplers $(\overline{P}_\alpha(Q))_{\alpha \in [0,1]}$ over the original $P$. Within this family and in the positive-semi-definite case, we see that the speedup measured in terms of $\lambda_2$ is at most one half, and as such one may seek to find an optimal $Q$ that minimizes
$\lambda_2(\overline{P}_\alpha(Q))$ subject to the constraints $Q^* = Q$ and $Q^2 = I$. We shall discuss tuning strategies of $Q$ in Section \ref{sec:tuneQ}.

Another interesting consequence of Proposition \ref{prop:spectralspeedup} lies in the equality characterizations. Under what situation(s) are we guaranteed to have $\lambda_2(\overline{P}_\alpha(Q)) < \lambda_2(P)$? Suppose that $P$ is ergodic, $\pi$-reversible with distinct eigenvalues (such as birth-death processes), and hence both the algebraic and geometric multiplicity equal to $1$ for each eigenvalue of $P$. Suppose that there exists $x \neq y$ such that $\pi(x) = \pi(y)$. We define $\phi(x) := y, \phi(y) := x, \phi(z) := z$ for all $z \in \mathcal{X} \backslash \{x,y\}$. Define $Q_{\phi}(x,y) = \delta_{y = \phi(x)}$, the Dirac mass of the set $\{y = \phi(x)\}$. Then, a necessary condition for $\lambda_2(\overline{P}_\alpha(Q_\phi)) = \lambda_2(P)$ is that the common eigenvector $f$ satisfies $f(x) = \pm f(y)$. Thus, under these assumptions of $Q_\phi$, if the eigenvector of $P$ has distinct absolute values such that $|f(x)| \neq |f(y)|$ for all $x \neq y$, the Weyl's inequality is strict and hence $\lambda_2(\overline{P}_\alpha(Q_\phi)) < \lambda_2(P)$. Similar analysis can be done to give a sufficient condition for $\lambda_n(P) < \lambda_n(\overline{P}_{\alpha}(Q_\phi))$.

Analogous to Corollary \ref{cor:optimal1/2}, an interesting corollary of Proposition \ref{prop:spectralspeedup} is that the choice of $\alpha = 1/2$ is optimal within the family $(\overline{P}_\alpha(Q))_{\alpha \in [0,1]}$ in the sense that it minimizes SLEM and $\lambda_2$ when $P$ is $\pi$-reversible:

\begin{corollary}[Optimality of $\alpha = 1/2$]\label{cor:optimal1/22}
	Let $P \in \mathcal{L}(\pi)$ and $Q \in \mathcal{I}(\pi) \cap \mathcal{L}$ be an isometric involution transition matrix. We have
	\begin{align*}
		\min_{\alpha \in [0,1]} \mathrm{SLEM}(\overline{P}_\alpha(Q)) &= \mathrm{SLEM}(\overline{P}(Q)), \\
		\min_{\alpha \in [0,1]} \lambda_2(\overline{P}_\alpha(Q)) &= \lambda_2(\overline{P}(Q)).
	\end{align*}
\end{corollary}

The proof can readily be seen from Proposition \ref{prop:spectralspeedup} by replacing $P$ therein by $\overline{P}_\alpha(Q)$ as well as \eqref{eq:Pbarbar}.

\subsection{Comparisons of asymptotic variances}

In addition to entropic and spectral parameters, another commonly used parameter to assess the convergence of Markov chain samplers is asymptotic variance. In this subsection, we compare the asymptotic variances of various Markov chains. We first fix a few notations.

For an ergodic $P \in \mathcal{S}(\pi)$, its fundamental matrix $Z(P)$, see for example \cite[Chapter $6$]{B99}, is defined to be
\begin{align*}
	Z(P) := (I - (P - \Pi))^{-1},
\end{align*}
where we recall that $\Pi$ is the matrix where each row equals to $\pi$. Note that the above inverse always exists for ergodic $P$. For $Q \in \mathcal{I}(\pi) \cap \mathcal{L}$, we see that
\begin{align*}
	Z(QPQ) = QZ(P)Q.
\end{align*}

Let $(X_n)_{n \geq 0}$ be the Markov chain with ergodic transition matrix $P$. Its asymptotic variance of $f \in \ell^2_0(\pi)$ is, for any initial distribution $\mu$,
\begin{align*}
	\lim_{n \to \infty} \dfrac{1}{n} \mathrm{Var}_\mu\left(\sum_{i=1}^n f(X_i)\right) = 2 \langle f,Z(P)f \rangle_\pi - \langle f,f \rangle_\pi =: v(f,P). 
\end{align*}
For a proof of the above expression one can consult \cite[Theorem $6.5$]{B99}. From this definition we readily check that 
\begin{align}\label{eq:varfvarqf}
	v(f,P) = v(Qf, QPQ).
\end{align} 
A useful variational characterization of asymptotic variance for $P \in \mathcal{L}(\pi)$ \cite{S18} is given by
\begin{align}\label{eq:asympvvar}
	v(f,P) = \sup_{g \in \ell^2_0(\pi)} 4 \langle f,g \rangle_\pi - 2 \langle (I-P) g, g \rangle_\pi - \langle f,f \rangle_\pi.
\end{align}
The worst-case asymptotic variance, studied for example in \cite{FSHS93}, is 
\begin{align}\label{eq:worstasympvar}
	V(P) := \sup_{f \in \ell^2_0(\pi), \norm{f}_{\pi} = 1} v(f,P) = \dfrac{1 + \lambda_2(P)}{1-\lambda_2(P)},
\end{align}
while the average-case asymptotic variance, investigated in \cite{CCHP12}, is
\begin{align}\label{eq:averageasympvar}
	\overline{v}(P) := \int_{f \in \ell^2_0(\pi), \norm{f}_{\pi} = 1} v(f,P) d S(f),
\end{align}
where $dS(f)$ is the uniform measure on the normalized surface area.

Our first proposition compares the asymptotic variances between $\overline{P}_\alpha(Q)$ and $P$. 

\begin{proposition}\label{prop:compareasympv}
	Let $P \in \mathcal{L}(\pi)$ be ergodic and $Q \in \mathcal{I}(\pi) \cap \mathcal{L}$ be an isometric involution transition matrix. For $\alpha \in [0,1]$, recall that $\overline{P}_{\alpha}(Q)$ is introduced in \eqref{eq:PbaralphaQ}. For any $f \in \ell^2_0(\pi)$, we have
	\begin{align*}
		v(f,\overline{P}_\alpha(Q)) \leq \alpha v(f,P) + (1-\alpha) v(Qf,P).
	\end{align*}
	In particular, if $f$ satisfies $Qf = \pm f$, it leads to
	\begin{align*}
		v(f,\overline{P}_\alpha(Q)) \leq v(f,P),
	\end{align*}
	and hence
	\begin{align*}
		\min_{\alpha \in [0,1]} v(f,\overline{P}_\alpha(Q)) &= v(f,\overline{P}(Q)).
	\end{align*}
\end{proposition}

Our second result compares the worst-case and average-case asymptotic variance between $\overline{P}_\alpha(Q)$ and $P$, and demonstrates the optimality of $\alpha = 1/2$.

\begin{proposition}\label{prop:asympvarspeedup}
	Let $P \in \mathcal{L}(\pi)$ and $Q \in \mathcal{I}(\pi) \cap \mathcal{L}$ be an isometric involution transition matrix. Fix $\alpha \in (0,1)$ and recall that $\overline{P}_{\alpha}(Q)$ is introduced in \eqref{eq:PbaralphaQ}. We have 
	\begin{itemize}
		\item(worst-case asymptotic variance) \begin{align*}
			V(\overline{P}_\alpha(Q)) \leq V(P), 
		\end{align*}
		where the equality holds if and only if $\lambda_2(\overline{P}_\alpha(Q)) = \lambda_2(P)$ if and only if there exists a common eigenvector $g$ such that $\overline{P}_{\alpha}(Q)g = \lambda_2(\overline{P}_{\alpha}(Q))g$, $Pg = \lambda_2(P) g$ and $QPQg = \lambda_2(QPQ) g$.
		
		If $P$ is further assumed to be positive-semi-definite, then
		\begin{align*}
				\max\{\alpha,1-\alpha\} V(P) &\leq V(\overline{P}_\alpha(Q)) \leq V(P).
		\end{align*}
		
		\item(average-case asymptotic variance) \begin{align*}
			\overline{v}(\overline{P}_\alpha(Q)) \leq \overline{v}(P).
		\end{align*}
	\end{itemize}
	Consequently, this leads to
	\begin{align*}
		\min_{\alpha \in [0,1]} V(\overline{P}_\alpha(Q)) &= V(\overline{P}(Q)), \\
		\min_{\alpha \in [0,1]} \overline{v}(\overline{P}_\alpha(Q)) &= \overline{v}(\overline{P}(Q)).
	\end{align*}
%	If $P$ (and hence $QPQ$) is assumed to be positive-semi-definite, then
%	\begin{align*}
%		\max\{\alpha,1-\alpha\} \lambda_2(P) &\leq \lambda_2(\overline{P}_{\alpha}(Q)) \leq \lambda_2(P), \\
%		\max\{\alpha,1-\alpha\} \mathrm{SLEM}(P) &\leq \mathrm{SLEM}(\overline{P}_\alpha(Q)) \leq \mathrm{SLEM}(P).
%	\end{align*}
\end{proposition}

When will there be no improvement in the worst-case asymptotic variance, i.e. $V(\overline{P}_\alpha(Q)) = V(P)$? One interesting consequence of Proposition \ref{prop:asympvarspeedup} is that, there is no improvement if and only if $\lambda_2(P) = \lambda_2(\overline{P}_\alpha(Q))$ and by the Weyl's inequality if and only if there exists a common eigenvector.

\section{Alternating projections to combine $Q$s}\label{sec:altproj}

Let $m \in \mathbb{N}$ and suppose that we have a sequence of isometric involution matrices $Q_i \in \mathcal{I}(\pi) \cap \mathcal{L}$ for $i \in \llbracket 0,m-1 \rrbracket$. Is there a way to combine these $Q_i$ to further improve the convergence to equilibrium? 

One natural idea in this context is alternating projections. Specifically, given a $P \in \mathcal{L}(\pi)$, we first project it onto the space $\mathcal{L}(\pi,Q_0)$ to obtain $R_1 = R_1(Q_0,\ldots,Q_{m-1},P) := \overline{P}(Q_0)$. Second, we project $R_1$ onto the space $\mathcal{L}(\pi,Q_1)$ to obtain $R_2 = R_2(Q_0,\ldots,Q_{m-1},P) := \overline{R_1}(Q_1)$. Third, we project $R_2$ onto the space $\mathcal{L}(\pi,Q_2)$ to obtain $R_3 = R_3(Q_0,\ldots,Q_{m-1},P) := \overline{R_2}(Q_2)$. We proceed iteratively and the projection order is deterministic in a cycle in the order of $Q_0,\ldots, Q_{m-1}$. Precisely, for $n \in \mathbb{N}$, we define
\begin{align}\label{eq:Rndef}
	R_n = R_n(Q_0,\ldots,Q_{m-1},P) := \overline{R_{n-1}}(Q_{(n-1) \mod m})
\end{align}
with the initial condition $R_0 := P$.

We remark that, in the context of MCMC, alternating projections have appeared in the analysis of Gibbs samplers in \cite{DKSC10,Q24}.

The sequence of alternating projections $(R_i)_{i \in \mathbb{N}}$ yields a monotone sequence of mixing time parameters. This can readily be seen by using the monotone convergence theorem, together with recursive application of Proposition \ref{prop:converge2}, \ref{prop:spectralspeedup}, \ref{prop:asympvarspeedup} and noting a lower bound of zero on these quantities.

\begin{proposition}\label{prop:alterprojlimit}
	Let $P \in \mathcal{L}(\pi)$ and $Q_i \in \mathcal{I}(\pi) \cap \mathcal{L}$ for $i \in \llbracket 0,m-1 \rrbracket$ be a sequence of isometric involution transition matrices. Define $R_n$ as in \eqref{eq:Rndef}. The sequences $(D^\pi_{KL}(R_n \| \Pi))_{n \in \mathbb{N}}$, $(c_{KL}(R_n))_{n \in \mathbb{N}}$, $(\mathrm{SLEM}(R_n))_{n \in \mathbb{N}}$, $(V(R_n))_{n \in \mathbb{N}}$ and $(\overline{v}(R_n))_{n \in \mathbb{N}}$ are monotonically non-increasing in $n$ with limits
	\begin{align*}
		\lim_{n \to \infty} D^\pi_{KL}(R_n \| \Pi) &= \inf_{n \in \mathbb{N}} D^\pi_{KL}(R_n \| \Pi),\\
		\lim_{n \to \infty} c_{KL}(R_n) &= \inf_{n \in \mathbb{N}} c_{KL}(R_n),\\
		\lim_{n \to \infty} \mathrm{SLEM}(R_n) &= \inf_{n \in \mathbb{N}} \mathrm{SLEM}(R_n),\\
		\lim_{n \to \infty} V(R_n) &= \inf_{n \in \mathbb{N}} V(R_n),\\
		\lim_{n \to \infty} \overline{v}(R_n) &= \inf_{n \in \mathbb{N}} \overline{v}(R_n),
	\end{align*}
	where we recall $D^\pi_{KL}$ is the $\pi$-weighted KL divergence, $c_{KL}$ is the KL-divergence Dobrushin coefficient as in Definition \eqref{def:KLDobrushin}, $\mathrm{SLEM}$ as in \eqref{def:SLEMP}, while $V$ and $\overline{v}$ are respectively the worst-case \eqref{eq:worstasympvar} and average-case \eqref{eq:averageasympvar} asymptotic variance.
\end{proposition}

Denote the intersections of $\mathcal{L}(\pi,Q_0), \ldots, \mathcal{L}(\pi,Q_{m-1})$ to be
\begin{align*}
	\mathcal{E} = \mathcal{E}(\pi,Q_0,\ldots,Q_{m-1}) := \bigcap_{k=0}^{m-1} \mathcal{L}(\pi,Q_k).
\end{align*}
Note that since $\Pi \in \mathcal{L}(\pi,Q_k)$ for all $k \in \llbracket 0,m-1 \rrbracket$, $\Pi \in \mathcal{E}$ and hence $\mathcal{E} \neq \emptyset$. Since $\mathcal{L}(\pi,Q_k)$ is a convex and compact set and intersections preserve convexity and compactness, $\mathcal{E}$ is also a convex and compact set. Let $R_\infty$ be an information projection of $P \in \mathcal{L}(\pi)$ onto $\mathcal{E}$, that is,
\begin{align*}
	R_\infty = R_\infty(Q_0,\ldots,Q_{m-1},P) := \argmin_{N \in \mathcal{E}} D^\pi_{KL}(P \| N).
\end{align*}
Observe that since for fixed $P$ the mapping $N \mapsto D^\pi_{KL}(P \| N)$ is convex (see e.g. \cite{M20}) and the above minimization is taken over a convex and compact set $\mathcal{E}$, a unique minimizer $R_\infty$ exists owing to the Pythagorean theorem (see e.g. \cite[Lemma $13.2.3$]{B17}).

We state a decomposition of the KL divergence and squared-Frobenius norm, and use it to show that the information projection of any $R_k$ onto $\mathcal{E}$ coincides and are equal to $R_\infty$:

\begin{proposition}\label{prop:usefuldecomp}
	Let $l,m,n \in \mathbb{N}\cup \{0\}$ with $l > n \geq 0$ and $m \geq 1$. Let $P \in \mathcal{L}(\pi)$ and $Q_i \in \mathcal{I}(\pi) \cap \mathcal{L}$ for $i \in \llbracket 0,m-1 \rrbracket$ be a sequence of isometric involution transition matrices. Define $R_n$ as in \eqref{eq:Rndef}. 
	\begin{itemize}
		\item For $E \in \mathcal{E}$, we have
		\begin{align*}
			D^\pi_{KL}(R_n \| E) 
			&= \sum_{j=n}^{l-1} D^\pi_{KL}(R_j \| R_{j+1}) + D^\pi_{KL}(R_{l} \| E).
		\end{align*}
	
		\item For $E \in \mathcal{E}$, we have
		\begin{align*}
			\norm{R_n - E}_F^2 
			&= \sum_{j=n}^{l-1} \norm{R_j - R_{j+1}}_F^2 + \norm{R_{l} - E}_F^2.
		\end{align*}
		
		\item  We have
		\begin{align*}
			R_\infty(Q_0,\ldots,Q_{m-1},P) = R_\infty(Q_0,\ldots,Q_{m-1},R_l).
		\end{align*}
	
		\item(Projection is trace-preserving)
		\begin{align*}
			\mathrm{Tr}(R_n) = \mathrm{Tr}(P).
		\end{align*}
	\end{itemize}
\end{proposition}

\begin{proof}
	Let $Q = Q_{n \mod m}$. For the first item, we repeatedly apply Proposition \ref{prop:deformKLorigKL} and the Pythagorean identity in Proposition \ref{prop:basic} to obtain
	\begin{align*}
		D^\pi_{KL}(R_n \| E) &= D^{Q}_{KL}(R_n \| E) \\
		&= D^Q_{KL}(R_n \| R_{n+1}) + D^Q_{KL}(R_{n+1} \| E) \\
		&= D^\pi_{KL}(R_n \| R_{n+1}) + D^\pi_{KL}(R_{n+1} \| E) \\
		&= D^\pi_{KL}(R_n \| R_{n+1}) + D^\pi_{KL}(R_{n+1} \| R_{n+2}) + D^\pi_{KL}(R_{n+2} \| E) \\
		&= \quad \vdots \\
		&= \sum_{j=n}^{l-1} D^\pi_{KL}(R_j \| R_{j+1}) + D^\pi_{KL}(R_{l} \| E).
	\end{align*}
	Now, we take $n = 0$ above and since $D^\pi_{KL}(R_0 \| E) \geq D^\pi_{KL}(R_0 \| R_\infty)$, we are led to $D^\pi_{KL}(R_l \| E) \geq D^\pi_{KL}(R_l \| R_\infty)$, and hence $R_\infty(Q_0,\ldots,Q_{m-1},P) = R_\infty(Q_0,\ldots,Q_{m-1},R_l)$, which proves the third item.
	
	For the second item, we repeatedly apply Proposition \ref{prop:projectFrob} to obtain
	\begin{align*}
		\norm{R_n - E}_F^2 
		&= \norm{R_n - R_{n+1}}_F^2 + \norm{R_{n+1} - E}_F^2 \\
		&= \norm{R_n - R_{n+1}}_F^2 + \norm{R_{n+1} - R_{n+2}}_F^2 + \norm{R_{n+2} - E}_F^2 \\
		&= \quad \vdots \\
		&= \sum_{j=n}^{l-1} \norm{R_j - R_{j+1}}_F^2 + \norm{R_{l} - E}_F^2.
	\end{align*}

	Finally, we apply Proposition \ref{prop:converge} repeatedly to yield $\mathrm{Tr}(R_n) = \mathrm{Tr}(P)$.
\end{proof}

Our main result shows that the limit of $R_n$ exists and is given by $R_\infty$.

\begin{theorem}\label{thm:alterprojmain}
	Let $m,n \in \mathbb{N}$. Let $P \in \mathcal{L}(\pi)$ and $Q_i \in \mathcal{I}(\pi) \cap \mathcal{L}$ for $i \in \llbracket 0,m-1 \rrbracket$ be a sequence of isometric involution transition matrices. Define $R_n$ as in \eqref{eq:Rndef}. The following limit exists (pointwise or in total variation):
	\begin{align*}
		\lim_{n \to \infty} R_n = R_\infty,
	\end{align*}
	and
	\begin{align*}
		R_\infty \in \mathcal{L}(\pi) \cap \mathcal{E}, \quad \mathrm{Tr}(R_\infty) = \mathrm{Tr}(P).
	\end{align*}
\end{theorem}

\begin{proof}
	The proof is inspired by that of \cite[Theorem $3.2$]{C75}. On a finite state space $\mathcal{X}$, it suffices to show that for every converging subsequence $(R_{n_k})_{k \geq 1}$ such that $R_{n_k} \to R^{\prime}$, we have $R^{\prime} = R_\infty$.
	
	For $M,N \in \mathcal{L}$ and $\pi \in \mathcal{P}(\mathcal{X})$, we define the $\pi$-weighted total variation distance between $M$ and $N$ to be
	\begin{align*}
		D^\pi_{TV}(M,N) := \sum_x \pi(x) \sum_y |M(x,y) - N(x,y)|.
	\end{align*}
	
	First, we show that $R^{\prime} \in \mathcal{E}$. Taking $n = 0$ in Proposition \ref{prop:usefuldecomp} and using Proposition \ref{prop:alterprojlimit} give us that
	\begin{align*}
		\lim_{l \to \infty} \sum_{j=0}^{l-1} D^\pi_{KL}(R_j \| R_{j+1}) = \lim_{l \to \infty} D^\pi_{KL}(P \| E) - D^\pi_{KL}(R_l \| E) = D^\pi_{KL}(P \| E) - \inf_{l \in \mathbb{N}} D^\pi_{KL}(R_l \| E) < \infty, 
	\end{align*}
	and hence $\lim_{l \to \infty} D^\pi_{KL}(R_l \| R_{l+1}) = 0$. By the Markov chain Pinsker's inequality \cite[Proposition $3.5$]{WC23}, we thus have $\lim_{l \to \infty} D^\pi_{TV}(R_l , R_{l+1}) = 0$. By the triangle inequality, the subsequences $(R_{n_k})_{k \geq 1}$, $(R_{n_k + 1})_{k \geq 1}$, $(R_{n_k+2})_{k \geq 1}$ up to $(R_{n_k + m})_{k \geq 1}$ all converge to $R^\prime$. Since the spaces $\mathcal{L}(\pi,Q_0), \ldots \mathcal{L}(\pi,Q_{m-1})$ are closed, $R^{\prime} \in \mathcal{E}$. It is also obvious to note that $R^{\prime} \in \mathcal{L}(\pi)$ since $P \in \mathcal{L}(\pi)$.
	
	By the definition of $R_\infty$ and Proposition \ref{prop:usefuldecomp}, we have
	\begin{align*}
		D^\pi_{KL}(R_{n_k} \| R^{\prime}) \geq D^\pi_{KL}(R_{n_k} \| R_\infty).
	\end{align*}
	Taking $k \to \infty$ yields
	\begin{align*}
		\lim_{k \to \infty} D^\pi_{KL}(R_{n_k} \| R_\infty) = D^\pi_{KL}(R^\prime \| R_\infty) = 0,
	\end{align*}
	and hence $R^\prime = R_\infty$. 
	
	Finally, to prove that $\mathrm{Tr}(R_\infty) = \mathrm{Tr}(P)$, we take the limit $n \to \infty$ in item 3 of Proposition \ref{prop:usefuldecomp}. Since $\mathcal{X}$ is finite and $R_n$ converges to $R_\infty$ pointwise, this completes the proof.
%	Let $m,n \in \mathbb{N}$ and without loss of generality assume $m > n$. We repeatedly apply the Pythagorean inequality in Proposition \ref{prop:converge2} to obtain
%	\begin{align*}
%		D^\pi_{KL}(R_n \| \Pi) &\geq D^\pi_{KL}(R_n \| R_{n+1}) + D^\pi_{KL}(R_{n+1} \| \Pi) \\
%		&\geq D^\pi_{KL}(R_n \| R_{n+1}) + D^\pi_{KL}(R_{n+1} \| R_{n+2}) + D^\pi_{KL}(R_{n+2} \| \Pi) \\
%		&\geq \quad \vdots \\
%		&\geq \sum_{k=n}^{m-1} D^\pi_{KL}(R_k \| R_{k+1}) + D^\pi_{KL}(R_{m} \| \Pi).
%	\end{align*}
%	Taking square on both sides and using $(a+b)^2 \geq a^2 + b^2$ for non-negative real numbers $a,b \geq 0$ lead to
%	\begin{align*}
%		D^\pi_{KL}(R_n \| \Pi)^2 &\geq \sum_{k=n}^{m-1} D^\pi_{KL}(R_k \| R_{k+1})^2 + D^\pi_{KL}(R_{m} \| \Pi)^2.
%	\end{align*}
%	Taking $n,m \to \infty$ and noting Proposition \ref{prop:alterprojlimit} to yield
%	\begin{align*}
%		\lim_{n,m \to \infty} \sum_{k=n}^{m-1} D^\pi_{KL}(R_k \| R_{k+1}) = 0.
%	\end{align*}
%	
%	Note that $\mathcal{L}$ endowed with the metric $D^\pi_{TV}$ is a complete metric space. Using the Pinsker's inequality \textcolor{black}{REF} and triangle inequality gives
%	\begin{align*}
%		\sum_{k=n}^{m-1} D^\pi_{KL}(R_k \| R_{k+1}) \geq 2 \sum_{k=n}^{m-1} D^\pi_{TV}(R_k , R_{k+1})^2,
%	\end{align*}
%	and hence
%	\begin{align*}
%		\lim_{n,m \to \infty} \sum_{k=n}^{m-1} D^\pi_{TV}(R_k , R_{k+1})^2 = 0
%	\end{align*}
\end{proof}

$R_\infty$ can therefore be interpreted as an ``optimal" transition matrix that combines $Q_0,\ldots,Q_{m-1}$. The following Figure \ref{fig:altproj} visualizes the alternation projection procedure in the case of $m = 2$.

\begin{figure}[H]
	\centering
	\includegraphics[width=0.75\textwidth]{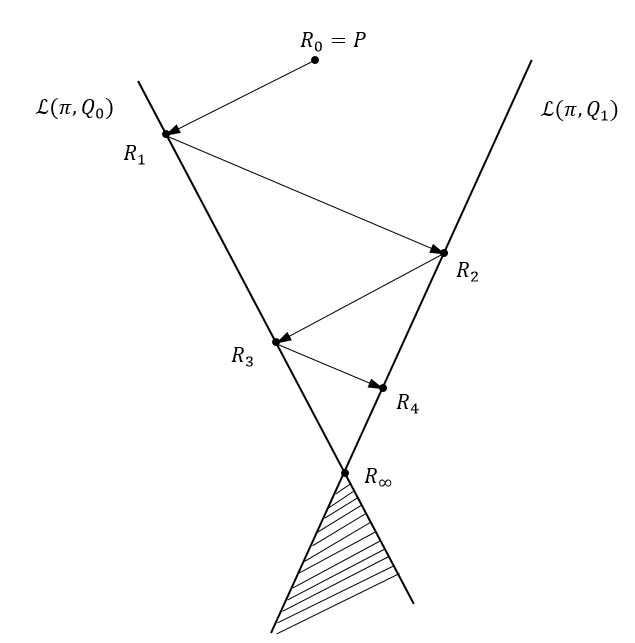}
	\caption{Improving the mixing of $P$ via alternating projections with $m = 2$. The intersection $\mathcal{E} = \cap_{i=0}^1 \mathcal{L}(\pi,Q_i)$ is the striped region in the bottom, and $R_\infty$ is the projection of $R_0 = P$ onto $\mathcal{E}$ under $D^\pi_{KL}$.}
	\label{fig:altproj}
\end{figure}

In using these alternating projections to improve mixing, the worst possible case that can happen is that $P = R_\infty$, that is, using these alternating projections have no effect on the original $P$. This happens if and only if $P \in \mathcal{E}$, that is, the original $P$ is $(\pi,Q_i)$-self-adjoint for all $i \in \llbracket 0,m-1 \rrbracket$.

On the other hand, an ideal scenario is that $R_\infty = \Pi$, and by Proposition \ref{prop:usefuldecomp} that happens if and only if
\begin{align*}
	D^\pi_{KL}(P \| \Pi) 
	&= \sum_{j=0}^{\infty} D^\pi_{KL}(R_j \| R_{j+1}).
\end{align*}
We shall investigate some other necessary conditions of $R_\infty = \Pi$ in Section \ref{sec:maxspeedlimit}.

A special case appears when $(Q_i)_{i=0}^{m-1}$ is pairwise commutative, that is, $Q_i Q_j = Q_j Q_i$ for all $i \neq j, i,j \in \llbracket 0,m-1 \rrbracket$. Define the product to be
\begin{align*}
	\mathbf{Q} = \mathbf{Q}(\pi,Q_0,\ldots,Q_{m-1}) := \prod_{i=0}^{m-1} Q_i.
\end{align*}
Using the pairwise commutative property one can verify that $\mathbf{Q}^2 = I$, $\mathbf{Q}^* = \mathbf{Q}$ and $R_n = R_m$ for all $n \geq m$. It can also be seen that
\begin{align*}
	\mathcal{E} = \mathcal{L}(\pi,\mathbf{Q}),
\end{align*}
and hence 
\begin{align*}
	R_\infty = \overline{P}(\mathbf{Q}) = \dfrac{1}{2}(P + \mathbf{Q}P\mathbf{Q}) = R_m.
\end{align*}

\subsection{A recursive simulation procedure for $(R_n)_{n \in \mathbb{N}}$}

In view of Proposition \ref{prop:characterizIpi}, let $(\psi_i)_{i=0}^{m-1}$ be a sequence of permutations with $\psi_i \in \Psi(\pi)$ and we take $Q_i = Q_{\psi_i}$ to be the induced permutation matrices. In this subsection, we devise a recursive simulation procedure for $R_n$ assuming $R_0 = P \in \mathcal{L}(\pi)$.

First, to simulate $R_1 = (1/2)(P + Q_0 P Q_0)$ is straightforward: with probability $1/2$ we either use $P$ or $Q_0 P Q_0$. Let $\sigma_1$ be a random permutation defined to be either the identity map $\mathbf{i}$ or $\psi_0$ with equal probability. That is,
\begin{align*}
	\sigma_1 := \begin{cases}
		\mathbf{i}, \quad \textrm{with probability 1/2},\\
		\psi_0, \quad \textrm{with probability 1/2}.\\
	\end{cases}
\end{align*}
Thus, to simulate one step of $R_1$ with an initial state $x$ to a state $y$, it is equivalent to simulate $y^\prime \sim P(\sigma_1(x),\cdot)$ followed by setting $y = \sigma_1(y^\prime)$.

Building upon $R_1$, we simulate $R_2 = (1/2)(R_1 + Q_{1 \mod m} R_1 Q_{1 \mod m})$. Let $\sigma_2$ be a random permutation defined to be either the identity map $\mathbf{i}$ or $\psi_{1 \mod m}$ with equal probability. That is, 
\begin{align*}
	\sigma_2 := \begin{cases}
		\mathbf{i}, \quad \textrm{with probability 1/2},\\
		\psi_{1 \mod m}, \quad \textrm{with probability 1/2}.\\
	\end{cases}
\end{align*}
Thus, to simulate one step of $R_2$ with an initial state $x$ to a state $y$, it is equivalent to simulate $y^\prime \sim P(\sigma_1(\sigma_2(x)),\cdot)$ followed by setting $y = \sigma_2(\sigma_1(y^\prime))$.

Continuing the above construction recursively or by induction, we simulate $R_n = (1/2)(R_{n-1} +Q_{n-1 \mod m} R_{n-1} Q_{n-1 \mod m})$. Let $\sigma_n$ be a random permutation defined to be either the identity map $\mathbf{i}$ or $\psi_{n-1 \mod m}$ with equal probability. That is,
\begin{align*}
	\sigma_n := \begin{cases}
		\mathbf{i}, \quad \textrm{with probability 1/2},\\
		\psi_{n-1 \mod m}, \quad \textrm{with probability 1/2}.\\
	\end{cases}
\end{align*}
Thus, to simulate one step of $R_n$ with an initial state $x$ to a state $y$, it is equivalent to simulate $y^\prime \sim P(\sigma_1 \circ \ldots \circ \sigma_n (x),\cdot)$ followed by setting $y = \sigma_n \circ \ldots \circ \sigma_1(y^\prime)$.

In summary, we first simulate realizations of the permutations $\sigma_1,\ldots,\sigma_n$. Starting from an initial state $x$, we draw a random $y^{\prime} \sim P(\sigma_1 \circ \ldots \circ \sigma_n (x),\cdot)$, then we set $y = \sigma_n \circ \ldots \circ \sigma_1(y^\prime)$. This simulates a move from $x$ to $y$ using $R_n$.

In comparing the additional computational costs between $R_n$ and $P$, we see that, to simulate one step of $R_n$, one needs to apply permutations for $2n$ times, namely $\sigma_1 \circ \ldots \circ \sigma_n$ followed by $\sigma_n \circ \ldots \circ \sigma_1$.

\subsection{Angle between subspaces and the rate of convergence of $R_n$ towards $R_\infty$}

In practice, to compute $R_\infty$, we can only run the alternating projections up to a finite time $n$ and arrive at $R_n$. How far away is $R_n$ from $R_\infty$? One way to measure this is, by Proposition \ref{prop:usefuldecomp}, 
\begin{align*}
	D^\pi_{KL}(R_n \| R_\infty) 
	&= \sum_{j=n}^{\infty} D^\pi_{KL}(R_j \| R_{j+1}).
\end{align*}

We now apply the theory of alternating projections to obtain a rate of convergence. First, we recall the notation of Section \ref{subsec:projectfrob}, where we have $m$ closed subspaces  $(\mathcal{M}(\pi,Q_i))_{i=0}^{m-1}$ of the Hilbert space $(\mathcal{M}, \langle \cdot,\cdot \rangle_F)$. Denote the intersection to be
$\mathcal{F} = \mathcal{F}(\pi,Q_0,\ldots,Q_{m-1}) := \cap_{i=0}^{m-1}\mathcal{M}(\pi,Q_i)$.

For any two closed subspace $(\mathcal{M}_i)_{i=1}^2$ of $\mathcal{M}$, we define the cosine of the angle between these two subspaces \cite[Definition $9.4$]{D01} to be 
\begin{align*}
	\alpha(\mathcal{M}_1,\mathcal{M}_2) := \sup\{|\langle M_1, M_2 \rangle_F|;~ M_i \in \mathcal{M}_i \cap (\mathcal{M}_1 \cap \mathcal{M}_2)^{\perp}, \norm{M_i}_F \leq 1, i \in \{1,2\}\}.
\end{align*}
Consider $\alpha(\mathcal{M}(\pi,Q_i),\mathcal{M}(\pi,Q_j))$ for $i \neq j$. If $Q_i$ is different from (resp. similar to) $Q_j$, we expect the angle between the two subspaces to be large (resp. small), leading to a small (resp. large) $\alpha$. Thus, $\alpha$ in our context can be broadly understood as a measure of dissimilarity between two permutations.
 
Let $P \in \mathcal{L}(\pi) \subseteq \mathcal{M}$ and consider the sequence $(R_n)$ defined in \eqref{eq:Rndef}. Define the projection of $P$ onto $\mathcal{F}$ under the Frobenius norm to be
$$R_\infty^{\prime} = \argmin_{M \in \mathcal{F}} \norm{P - M}_F.$$

According to \cite[Corollary $9.28$]{D01}, we have
\begin{align*}
	\lim_{n \to \infty} \norm{R_{mn} -R_\infty^{\prime}}_F = 0.
\end{align*}
Since $R_n \to R_\infty$ pointwise as shown in Theorem \ref{thm:alterprojmain}, we thus have $R_\infty = R_\infty^{\prime}$.

Define for $i \in \{0,1,\ldots,r-2\}$ 
\begin{align}\label{eq:alpha}
	\alpha_i &:= \alpha(\mathcal{M}(\pi,Q_i),\cap_{j = i + 1}^{m-1}\mathcal{M}(\pi,Q_j)), \nonumber \\
	\alpha &:= \sqrt{1 - \prod_{i=0}^{r-2}(1-\alpha_i^2)}. 
\end{align}
Using \cite[Theorem $9.33$]{D01} we arrive at
\begin{corollary}
	Let $P \in \mathcal{L}(\pi) \subseteq \mathcal{M}$ and consider the sequence $(R_n)$ defined in \eqref{eq:Rndef}, where $Q_i \in \mathcal{I}(\pi) \cap \mathcal{L}$ for $i \in \llbracket 0,m-1 \rrbracket$ is a sequence of isometric involution transition matrices. We have
	\begin{align*}
		\norm{R_{mn} - R_\infty}_F \leq \alpha^n \norm{P}_F,
	\end{align*}
	where $\alpha$ is given by \eqref{eq:alpha}. 
	%Note that the rate $\alpha$ does not depend on the initial condition $R_0 = P$.
\end{corollary}

An implication of the above Corollary allows us to answer the following question: how many alternating projection steps $t$ one need to run until we are guaranteed that $\norm{R_{t} - R_\infty}_F \leq \varepsilon$ for a given error $\varepsilon > 0$? Using the crude bound that $\norm{P}_F \leq \sqrt{|\mathcal{X}|}$, we see that one can set 
$$t \geq m \dfrac{\log (\sqrt{|\mathcal{X}|}/\varepsilon)}{\log (1/\alpha)}.$$

\section{The ``maximum speed limit" of projection samplers}\label{sec:maxspeedlimit}

In this section, we explore necessary conditions of achieving $R_n = \Pi$ or $R_\infty = \Pi$.

\subsection{A simple three-point example that achieves $\overline{P}(Q) = \Pi$}

In this subsection, we let $\mathcal{X}$ be a three-point state space. The aim of this subsection is to provide a non-trivial example that achieves $\overline{P}(Q) = \Pi$. For $P \in \mathcal{L}(\mathcal{X})$, recall that $\lambda(P)$ is the set of eigenvalues of $P$. Let $P$ be a transition matrix given by
\begin{align*}
	P = \begin{bmatrix}
		\frac{1}{2} & \frac{1}{3} & \frac{1}{6} \\
		\frac{1}{3} & \frac{1}{6} & \frac{1}{2} \\
		\frac{1}{6} & \frac{1}{2} & \frac{1}{3}
	\end{bmatrix}, \quad \lambda(P) = \bigg\{1, \pm \frac{1}{2 \sqrt{3}}\bigg\}.
\end{align*}

Clearly, $\pi = (1/3, 1/3, 1/3)$ and $P^* = P$. We take $Q$ to be
\begin{align*}
	Q = \begin{bmatrix}
		0 & 1 & 0 \\
		1 & 0 & 0 \\
		0 & 0 & 1
	\end{bmatrix}, \quad \lambda(Q) = \bigg\{1, 1, -1\bigg\}.
\end{align*}
It can easily be seen that $Q^* = Q$ and $Q^2 = I$. We compute that
\begin{align*}
	QP &= \begin{bmatrix}
		\frac{1}{3} & \frac{1}{6} & \frac{1}{2} \\
		\frac{1}{2} & \frac{1}{3} & \frac{1}{6} \\
		\frac{1}{6} & \frac{1}{2} & \frac{1}{3}
	\end{bmatrix}, \quad \lambda(QP) = \bigg\{1, \pm \frac{\sqrt{-3}}{6}\bigg\}, \\
	PQ &= \begin{bmatrix}
		\frac{1}{3} & \frac{1}{2} & \frac{1}{6} \\
		\frac{1}{6} & \frac{1}{3} & \frac{1}{2} \\
		\frac{1}{2} & \frac{1}{6} & \frac{1}{3}
	\end{bmatrix}, \quad \lambda(PQ) = \bigg\{1, \pm \frac{\sqrt{-3}}{6}\bigg\}, \\
	QPQ &= \begin{bmatrix}
		\frac{1}{6} & \frac{1}{3} & \frac{1}{2} \\
		\frac{1}{3} & \frac{1}{2} & \frac{1}{6} \\
		\frac{1}{2} & \frac{1}{6} & \frac{1}{3}
	\end{bmatrix}, \quad \lambda(QPQ) = \bigg\{1, \pm \frac{1}{2 \sqrt{3}}\bigg\}, \\
	\overline{P}(Q) &= \frac{1}{2}(P + QPQ) = \begin{bmatrix}
		\frac{1}{3} & \frac{1}{3} & \frac{1}{3} \\
		\frac{1}{3} & \frac{1}{3} & \frac{1}{3} \\
		\frac{1}{3} & \frac{1}{3} & \frac{1}{3}
	\end{bmatrix}, \quad \lambda(\overline{P}(Q)) = \bigg\{1, 0, 0\bigg\}, \\
	\frac{1}{2}(PQ + QP) &= \begin{bmatrix}
		\frac{1}{3} & \frac{1}{3} & \frac{1}{3} \\
		\frac{1}{3} & \frac{1}{3} & \frac{1}{3} \\
		\frac{1}{3} & \frac{1}{3} & \frac{1}{3}
	\end{bmatrix}, \quad \lambda\left(\frac{1}{2}(PQ + QP)\right) = \bigg\{1, 0, 0\bigg\}. 
\end{align*}

We see that there are two interesting properties of $P$: it satisfies $\mathrm{Tr}(P) = 1$ and $\lambda(P) \cap \lambda(-P) \neq \emptyset$. This motivates our investigations in the following subsections.

\subsection{A necessary condition of $R_n = \Pi$ in terms of trace}

In view of Proposition \ref{prop:usefuldecomp} and \ref{prop:converge}, we recall that the projections are trace-preserving and hence
$$\mathrm{Tr}(R_n) = \mathrm{Tr}(P).$$
Thus, if for some $n \in \mathbb{N} \cup \{\infty\}$ such that $R_n = \Pi$, this implies $\mathrm{Tr}(P) = 1$. We record this as a Corollary:

\begin{corollary}\label{cor:necessaryofRnPi}
	Let $P \in \mathcal{S}(\pi)$ and consider the sequence $(R_n)$ defined in \eqref{eq:Rndef}, where $Q_i \in \mathcal{I}(\pi) \cap \mathcal{L}$ for $i \in \llbracket 0,m-1 \rrbracket$ is a sequence of isometric involution transition matrices. If $R_n = \Pi$ for some $n \in \mathbb{N} \cup \{\infty\}$, then
	$$\mathrm{Tr}(P) = 1.$$
\end{corollary}

Consequently, this implies that if $P$ is positive-definite so that $\mathrm{Tr}(P) > 1$, then for any sequence of $(Q_i)_{i=0}^{m-1}$, $R_n \neq \Pi$ for all $n \in \mathbb{N} \cup \{\infty\}$.

%Let $c := \mathrm{Tr}(P)$ and assume that $c > 0$. Let $a,b$ be chosen such that $ac + nb = 1$ and $a + b = 1$. This yields Then, one can instead consider the transition matrix
%\begin{align*}
%	a P + b I, \quad \mathrm{Tr
%\end{align*}

\subsection{A necessary condition of $\overline{P}(Q) = \Pi$ via the Sylvester's equation}

Let us first briefly recall the Sylvester's equation. It is a linear matrix equation in $X \in \mathcal{M}$ of the form, for given $A,B,C \in \mathcal{M}$,
\begin{align*}
	AX + XB = C.
\end{align*}
The Sylvester's theorem \cite[Theorem $2.4.4.1$]{HJ13} gives a necessary and sufficient condition for the above Sylvester's equation to admit a unique solution in $X \in \mathcal{M}$ for each given $C$: $X$ is unique if and only if $\lambda(A) \cap \lambda(-B) = \emptyset$, that is, $A$ and $-B$ have no eigenvalue in common.

In our setting, we specialize into $A = P, B = P^*$ and $C = 2\Pi$ with $P \in \mathcal{S}(\pi)$. We note that $X = \Pi$ is always a solution. Thus, if $\overline{P}(Q) = \Pi$, the Sylvester's equation has at least two solutions $X \in \{\Pi,Q\}$, and hence by the Sylvester's theorem we have $\lambda(P) \cap \lambda(-P^*) \neq \emptyset$. We record this as a Corollary:

\begin{corollary}
	Let $P \in \mathcal{S}(\pi)$ and $Q \in \mathcal{I}(\pi) \cap \mathcal{L}$ be an isometric involution transition matrix. If $\overline{P}(Q) = \Pi$, then
	$$\lambda(P) \cap \lambda(-P^*) \neq \emptyset,$$
	that is, $P$ and $-P^*$ have at least one common eigenvalue.
\end{corollary}

Consequently, the above result implies that, for $\pi$-reversible $P \in \mathcal{L}(\pi)$, if it is positive-definite or if $|\lambda_i(P)| \neq |\lambda_j(P)|$ for all $i \neq j$, then $P$ and $-P^* = -P$ have no common eigenvalue, and hence $\overline{P}(Q) \neq \Pi$.

\subsection{Characterization of $R_\infty$ when $\pi$ is the discrete uniform distribution, and a necessary and sufficient condition of $R_\infty = \Pi$}

In this subsection, we let $n = |\mathcal{X}|$ and consider $\pi$ to be the discrete uniform distribution with $P = P^T \in \mathcal{L}(\pi)$. Without loss of generality we assume the state space is of the form $\mathcal{X} = \llbracket n \rrbracket$. For $j \in \llbracket 2,n \rrbracket$, we define the permutations $(\psi_{1,j})_{j=2}^n$ to be $\psi_{1,j}(1) = j$, $\psi_{1,j}(j) = 1$ and $\psi_{1,j}(x) = x$ for all $x \in \mathcal{X}\backslash\{1,j\}$, and denote the induced permutation matrices by $(Q_{1,j})_{j=2}^n$. Clearly, $Q_{1,j} \in \mathcal{I}(\pi) \cap \mathcal{L}$, and we recall that the intersection of $(\mathcal{L}(\pi,Q_{1,j}))_{j=2}^{n}$ is written as
$$\mathcal{E} = \bigcap_{j=2}^n \mathcal{L}(\pi,Q_{1,j}).$$

In the above setting, the main result of this subsection characterizes $R_\infty$ and demonstrates that $R_\infty$ is a linear combination of $\Pi$ and $I$. It also proves that, under the choices of the permutation matrices $(Q_{1,j})_{j=2}^n$, $\mathrm{Tr}(P) = 1$ is necessary and sufficient to achieve $R_\infty = \Pi$:

\begin{theorem}\label{thm:characterizeRinfty}
	Let $\pi$ be the discrete uniform distribution on $\mathcal{X}$ and $P = P^T \in \mathcal{L}(\pi)$. Denote the sequence of isometric involution transition matrices to be $Q_{1,j}$ as in this subsection, and define the sequence of projections $(R_l)_{l \in \mathbb{N}}$ as in \eqref{eq:Rndef}. The limit $R_\infty$ is given by
	\begin{align}\label{eq:Rinftyform}
		R_\infty = (nb) \Pi + (a-b) I = \begin{bmatrix}
			a & b & \ldots & b \\
			b & a & \ddots & \vdots \\
			\vdots & \ddots & \ddots & b \\
			b & \ldots & b & a
		\end{bmatrix},
	\end{align}
	where $a = a(Q_{1,2},\ldots,Q_{1,n},P), b = b(Q_{1,2},\ldots,Q_{1,n},P) \in [0,1]$ satisfy $nb + a-b = 1$. In particular, $R_\infty = \Pi$ if and only if $\mathrm{Tr}(P) = 1$.
\end{theorem}

\begin{proof}
	First, we shall prove by induction on $k \in \llbracket n \rrbracket$ that the first $k$ rows and $k$ columns of $R_\infty$ is of the form
	\begin{align*}
		\begin{bmatrix}
			a & b & \ldots & b & \ldots \\
			b & a & \ddots & b & \ldots \\
			\vdots & \ddots & \ddots & \vdots & \ldots \\
			b & \ldots & b & a & \ldots \\
			\ldots & \ldots & \ldots & \ldots & \ldots
		\end{bmatrix},
	\end{align*}
	that is, $R_\infty(x,x) = a$ for all $x \in \llbracket k \rrbracket$ and $R_\infty(x,y) = b$ for all $x \neq y, x,y \in \llbracket k \rrbracket$.
	
	When $k = 1$, what we seek to prove obviously holds. When $k = 2$, since $R_\infty \in \mathcal{L}(\pi,Q_{1,2})$, we see that
	\begin{align*}
		R_\infty(1,1) = R_\infty(\psi_{1,2}(1),\psi_{1,2}(1)) = R_\infty(2,2), \\
		R_\infty(1,2) = R_\infty(\psi_{1,2}(1),\psi_{1,2}(2)) = R_\infty(2,1).
	\end{align*}
	Assume that the induction hypothesis holds for some $k$. Since $R_\infty \in \bigcap_{j=2}^k \mathcal{L}(\pi,Q_{1,j})$, the first $k+1$ rows and $k+1$ columns of $R_\infty$ can be written as
	\begin{align*}
		\begin{bmatrix}
			a & b & \ldots & b & c &\ldots \\
			b & a & \ddots & b & c &\ldots \\
			\vdots & \ddots & \ddots & \vdots & \vdots & \ldots \\
			b & \ldots & b & a & c & \ldots \\
			d & \ldots & d & d & e & \ldots \\
			\ldots & \ldots & \ldots & \ldots & \ldots & \ldots
		\end{bmatrix},
	\end{align*}
	where $c,d,e \in [0,1]$ are some constants. To see that $c = d$, we note $c = R_\infty(1,k+1) = R_\infty(\psi_{1,k+1}(1),\psi_{1,k+1}(k+1)) = R_\infty(k+1,1) = d$. Similarly, we have $a = e$ since $a = R_\infty(1,1) = R_\infty(\psi_{1,k+1}(1),\psi_{1,k+1}(1)) = R_\infty(k+1,k+1) = e$. Finally, $b = c$ since $b = R_\infty(1,2) = R_\infty(\psi_{1,k+1}(1),\psi_{1,k+1}(2)) = R_\infty(k+1,2) = d = c$. This completes the induction.
	
	By Corollary \ref{cor:necessaryofRnPi}, $\mathrm{Tr}(P) = 1$ is a necessary condition of $R_\infty = \Pi$. In the opposite direction, if $\mathrm{Tr}(P) = 1$, the trace-preserving property of projections in Proposition \ref{prop:usefuldecomp} gives
	$\mathrm{Tr}(R_\infty) = 1$. We see that
	\begin{align*}
		\mathrm{Tr}(R_\infty) = \mathrm{Tr}((nb) \Pi + (a-b) I) = nb + (a-b)n = na = 1,
	\end{align*}
	which gives $a = b = 1/n$, and hence $R_\infty = \Pi$.
\end{proof}

Using \eqref{eq:Rinftyform}, we see that the right spectral gap of $R_\infty$ is
\begin{align*}
	\gamma(R_\infty) = nb.
\end{align*}
In the following proposition, we give a lower bound of $1/2$ when $P$ is closer to $\Pi$ than to $I$ in Frobenius norm:

\begin{proposition}\label{prop:spectralgaplow}
	In the setting of Theorem \ref{thm:characterizeRinfty}, if $\norm{P - \Pi}_F \leq \norm{P - I}_F$ or equivalently
	\begin{align*}
		\mathrm{Tr}(P) \leq \dfrac{n+1}{2},
	\end{align*}
	then
	\begin{align*}
		\gamma(R_\infty) \geq \dfrac{1}{2}.
	\end{align*}
\end{proposition}

\begin{proof}
	By Proposition \ref{prop:usefuldecomp}, the assumption $\norm{P - \Pi}_F \leq \norm{P - I}_F$ implies $\norm{R_\infty - \Pi}_F \leq \norm{R_\infty - I}_F$. Using \eqref{eq:Rinftyform}, we compute that
	\begin{align*}
		\norm{R_\infty - \Pi}_F &= |a-b| \norm{\Pi - I}_F,\\
		\norm{R_\infty - I}_F &= |1-(a-b)| \norm{\Pi - I}_F.
	\end{align*}
	This leads to $|a-b| \leq |1-(a-b)|$, and hence $a-b \leq 1/2$. Since $1 - nb = a-b$, this yields $nb \geq 1/2$.
	
	To see the equivalence, we compute that
	\begin{align*}
		\norm{P - \Pi}_F^2 &= \sum_{x,y} \left(P(x,y) - \frac{1}{n}\right)^2 \\
		&= \sum_{x,y}P(x,y)^2 - \frac{2}{n}\sum_{x,y}P(x,y) + 1 \\
		&=\sum_{x,y}P(x,y)^2  - 1,\\
		\norm{P - I}_F^2 &= \sum_{x,y}(P(x,y) - \delta_{x = y})^2 \\
		&= \sum_{x,y}P(x,y)^2 - 2\sum_{x,y}P(x,y) \delta_{x = y} + \sum_{x,y} \delta_{x = y} \\
		&=\sum_{x,y}P(x,y)^2  - 2 \sum_x P(x,x) + n,   
	\end{align*}
	and hence $\norm{P - \Pi}_F \leq \norm{P - I}_F$ is equivalent to 
	\begin{align*}
		\sum_x P(x,x) \leq \dfrac{n+1}{2}.
	\end{align*}
\end{proof}

In the remaining of this subsection, we let $c = c(P) := \mathrm{Tr}(P)$. Suppose that $c \in [0,1)$, and by Theorem \ref{thm:characterizeRinfty}, we note that
\begin{align*}
	R_\infty(Q_{1,2},\ldots,Q_{1,n},P) \neq \Pi.
\end{align*}
Consider instead the transition matrix $P^{\prime}$ given by
\begin{align*}
	P^\prime := \alpha I + (1-\alpha) P, \quad \alpha := \dfrac{1-c}{n-c} \in [0,1],
\end{align*}
then $P^{\prime T} = P^\prime \in \mathcal{L}(\pi)$. Furthermore, $\mathrm{Tr}(P^\prime) = 1$, and hence by Theorem \ref{thm:characterizeRinfty} and Proposition \ref{prop:spectralgaplow} we have
\begin{align*}
	R_\infty(Q_{1,2},\ldots,Q_{1,n},P^\prime) = \Pi, \quad \gamma(R_\infty(Q_{1,2},\ldots,Q_{1,n},P^\prime)) \geq \dfrac{1}{2}.
\end{align*}
Thus, it is advantageous to consider first $P^\prime$ and then the sequence of alternating projections $(R_l)$ induced by $P^\prime$ to improve mixing over the original $P$.

More generally, if $P$ is such that $c \in [0, \frac{n+1}{2}] \backslash \{1\}$, then by Theorem \ref{thm:characterizeRinfty} this is not an ideal situation since the limit of the projections is
\begin{align*}
	R_\infty(Q_{1,2},\ldots,Q_{1,n},P) \neq \Pi.
\end{align*}
However, by Proposition \ref{prop:spectralgaplow}
\begin{align*}
	 \gamma(R_\infty(Q_{1,2},\ldots,Q_{1,n},P)) \geq \dfrac{1}{2}.
\end{align*}
In other words, for $P \in \{P \in \mathcal{L}(\pi);~\pi(x) = 1/n \textrm{ for all } x, \mathrm{Tr}(P) \leq \frac{n+1}{2}\}$, the limit $R_\infty$ induced from any member of this family mixes fast since it has a constant order relaxation time.

A special case arises when $c = 0$, or equivalently $P(x,x) = 0$ for all $x$. This leads to $a = 0, b = 1/(n-1)$. Even if $R_\infty \neq \Pi$, $R_\infty$ is an ``optimal reversible stochastic matrix" that minimizes the worst-case asymptotic variance in the sense of \cite[Remark 3]{FHY92}.

\section{Tuning strategies of $Q$}\label{sec:tuneQ}

In this paper, given a $P \in \mathcal{L}(\pi)$, we propose projection samplers such as $\overline{P}(Q)$ or more generally the sequence of alternating projections $(R_l)_{l \in \mathbb{N}}$ as improved variants compared with the original $P$. In these cases, the isometric involution transition matrix $Q \in \mathcal{I}(\pi) \cap \mathcal{L}$ can be understood as a parameter in these algorithms, and the improvement depends on the tuning of $Q$. For instance, the choice of $Q = I$ is always feasible, yet it leads to no improvement since $\overline{P}(I) = P$. On the other hand, we have seen in Section \ref{sec:maxspeedlimit} that depending on $P$ it might be possible to achieve $R_l = \Pi$ or $R_\infty = \Pi$ with suitable choices of $Q$s.

In this section, we explore some possible tuning strategies of $Q$.

\subsection{Tuning $Q$ via optimization and Markov chain assignment problems}

The first strategy seeks to find an optimal $Q$ that minimizes the discrepancy between $\overline{P}(Q)$ and $\Pi$ or more generally between $R_l$ and $\Pi$.

Precisely, we would like to find $Q$ that minimizes the $\pi$-weighted KL divergence or the squared-Frobenius norm for a given $P \in \mathcal{S}(\pi)$:
\begin{align*}
	Q_{*,KL} &= Q_{*,KL}(P) := \argmin_{Q \in \mathcal{I}(\pi)\cap\mathcal{L}} D^\pi_{KL}(\overline{P}(Q) \| \Pi), \\
	Q_{*,F}  &= Q_{*,F}(P) := \argmin_{Q \in \mathcal{I}(\pi)\cap\mathcal{L}} \norm{\overline{P}(Q) - \Pi}_F^2. 
\end{align*}
The above optimization problems may not be solved in realistic time frame in practice, since $\pi$ may involve normalization constant that is non-tractable. Fortunately, using the Pythagorean identities in Proposition \ref{prop:projectFrob} and \ref{prop:converge2}, we see that
\begin{align*}
	Q_{*,KL} &= \argmax_{Q \in \mathcal{I}(\pi)\cap\mathcal{L}} D^\pi_{KL}(P \| \overline{P}(Q) ) =  \argmax_{\psi \in \Psi(\pi)} D^\pi_{KL}(P \| \overline{P}(Q_\psi) ), \\
	Q_{*,F}  &= \argmax_{Q \in \mathcal{I}(\pi)\cap\mathcal{L}} \norm{P - \overline{P}(Q)}_F^2 = \argmax_{\psi \in \Psi(\pi)} \norm{P - \overline{P}(Q_\psi)}_F^2. 
\end{align*}
The rightmost maximization problems can be understood as Markov chain assignment problems constrained to choosing permutations within the set $\Psi(\pi)$. While in general assignment problems can be solved in polynomial time in $|\mathcal{X}|$ \cite{KV18}, this may still be computationally infeasible in practice since $|\mathcal{X}|$ might be exponentially large in many models of interest in the context of MCMC.

The above can be generalized to consider multidimensional Markov chain assignment problems. Specifically, we seek to solve, for $m, l \in \mathbb{N}$,
\begin{align*}
	&\argmin_{\psi_i \in \Psi(\pi), \, \forall i \in \llbracket 0,m-1 \rrbracket} D^\pi_{KL}(R_l(Q_{\psi_0},\ldots,Q_{\psi_{m-1}},P) \| \Pi) \\
	&= \argmax_{\psi_i \in \Psi(\pi), \, \forall i \in \llbracket 0,m-1 \rrbracket} \sum_{j=0}^{l-1} D^\pi_{KL}(R_j(Q_{\psi_0},\ldots,Q_{\psi_{m-1}},P) \| R_{j+1}(Q_{\psi_0},\ldots,Q_{\psi_{m-1}},P)) , \\
	&\argmin_{\psi_i \in \Psi(\pi), \, \forall i \in \llbracket 0,m-1 \rrbracket} \norm{R_l(Q_{\psi_0},\ldots,Q_{\psi_{m-1}},P) - \Pi}_F^2 \\
	&= \argmax_{\psi_i \in \Psi(\pi), \, \forall i \in \llbracket 0,m-1 \rrbracket} \sum_{j=0}^{l-1} \norm{R_j(Q_{\psi_0},\ldots,Q_{\psi_{m-1}},P) - R_{j+1}(Q_{\psi_0},\ldots,Q_{\psi_{m-1}},P)}_F^2,
\end{align*}
where the equalities follow from the Pythagorean identities in Proposition \ref{prop:projectFrob} and \ref{prop:converge2}. Note that in general multidimensional assignment problems are NP hard \cite{NTP14} to solve, and there are heuristics to solve these in practice such as the cross-entropy method.

We remark that, this technique of converting the original problem of minimization of KL divergence to a maximization problem is in the spirit of evidence lower bound (ELBO) in variational inference, see for example \cite{BKM17}.

The additional computational cost in this strategy, compared with simply running the baseline Markov chain with transition matrix $P$, mainly lies in the time it takes to (approximately) solve the assignment problem using either heuristics or integer programming solvers.

\subsection{Tuning $Q$ adaptively in a single run}\label{subsec:tuneQsingle}

Let $H : \mathcal{X} \to \mathbb{R}$ be a target Hamiltonian function, and $\pi_\beta$ be its associated Gibbs distribution at inverse temperature $\beta \geq 0$, that is, for $x \in \mathcal{X}$,
\begin{align*}
	\pi_\beta(x) := \dfrac{e^{-\beta H(x)}}{Z_\beta},
\end{align*}
where $Z_\beta := \sum_{x \in \mathcal{X}} e^{-\beta H(x)}$ is the normalization constant. Thus, for $\beta > 0$, we see that $\pi_\beta(x) = \pi_\beta(y)$ if and only if $H(x) = H(y)$.

The second tuning strategy lies in adjusting $Q$ adaptively on the fly as the algorithm progresses. Specifically, given a $P \in \mathcal{L}(\pi_\beta)$ such as the Metropolis-Hastings algorithm or the Gibbs sampler, we run a non-homogeneous and adaptive Markov chain with transition matrix at each time $l \in \mathbb{N}$ to be
\begin{align*}
	\dfrac{1}{2}(P + Q_{\psi_l} P Q_{\psi_l}),
\end{align*}
along with the initial condition $Q_{\psi_1} = I$. We record the trajectories of this adaptive Markov chain. At time $l \geq 2$, suppose the past trajectory is $\{x_0,x_1,\ldots,x_{l-1}\}$. We search for an ``equi-energy" pair that is not mapped in $Q_{\psi_{l-1}}$: if there exists $i,j$ such that $H(x_i) = H(x_j)$, $\psi_{l-1}(x_i) = x_i, \, \psi_{l-1}(x_j) = x_j$, then we update the permutation to $\psi_{l}(x_i) = x_j, \, \psi_l(x_j) = x_i, \, \psi_l(x) =\psi_{l-1}(x)$ for all $x \in \mathcal{X} \backslash \{x_i,x_j\}$. More generally, the permutation matrix can be updated in say every $k$ steps, rather than $k = 1$ step, in order to reduce computational cost. For example, the adaptive projection sampler updates in every $k = 50$ steps in Section \ref{sec:numerical}.

The additional computational cost in this strategy, compared with simply running the baseline Markov chain with transition matrix $P$, mainly lies in the time it takes to adaptively update the permutation matrix, that is, the time it takes to search for equi-probability matches at each step that are not previously mapped.

\subsection{Tuning $Q$ using an exploration chain in multiple runs}\label{subsec:tuneQexplore}

The third strategy uses an exploration Markov chain, such as the proposal chain in Metropolis-Hastings or the Metropolis-Hastings chain at high temperature, for $k \in \mathbb{N}$ times. Each time a permutation matrix $Q_l$ is generated as outlined in Section \ref{subsec:tuneQsingle} for $l \in \llbracket k \rrbracket$. Then, we combine this sequence of matrices $(Q_l)_{l=1}^k$ using alternating projections as discussed in Section \ref{sec:altproj}. This idea is inspired by the equi-energy sampler \cite{KZW06}.

The additional computational cost in this strategy, compared with simply running the baseline Markov chain with transition matrix $P$, mainly lies in the time it takes to simulate the $k$ exploratory Markov chains, as well as the cost to search for equi-probability matches to generate $(Q_l)_{l=1}^k$.

\subsection{Tuning $Q$ by leveraging the symmetry of the state space $\mathcal{X}$}

In some statistical physics models of interests, the state space has a natural symmetric structure. For example, in the models investigated in Section \ref{sec:numerical} below, $\mathcal{X}$ is either $\{-1,+1\}^d$ or $\{-1,0,+1\}^d$, along with the property that the Hamiltonian is symmetric in the sense that it satisfies $H(\mathbf{x}) = H(-\mathbf{x})$ for all $\mathbf{x} \in \mathcal{X}$. In this context, a natural equi-probability permutation with respect to $\pi_\beta$ is the map $\mathbf{x} \mapsto -\mathbf{x}$.

If $H$ is not symmetric, then one can consider a symmetrized Gibbs distribution $\mu_\beta$ given by
\begin{align*}
	\mu_\beta(\mathbf{x}) \propto e^{-\beta G(\mathbf{x})}, \quad G(\mathbf{x}) = \frac{1}{2}(H(\mathbf{x}) + H(-\mathbf{x})).
\end{align*}
One can then apply the permutation $\mathbf{x} \mapsto -\mathbf{x}$ to improve the convergence of Markov chains that target $\mu_\beta$, followed by possibly an importance sampling step to correct for the bias in order to sample from $\pi_\beta$. More generally, one can follow the line of work by \cite{Ying2024} to introduce the so-called group orbit average distribution when there are multiple symmetries in the state space. We leave this direction as future work.

\section{Application to Metropolis-Hastings}\label{sec:MH}

The aim of this section is to concretely illustrate and quantify the benefit of using the projection sampler $\overline{P}(Q)$ over the original $P$, when the latter is taken to be the transition matrix of the classical Metropolis-Hastings (MH) algorithm. We also present a simple model in which $\pi_\beta$ is a discrete bimodal distribution, where the relaxation time of the projection sampler, upon suitable choice of $Q$, is polynomial in $\beta$ and size of the state space while that of the MH is exponential in these parameters. Thus, the relaxation (and hence mixing) time is improved from exponential to polynomial via this technique.

To this end, let us briefly recall the MH dynamics. Given a proposal Markov chain with transition matrix $N$ that is ergodic and reversible, and a Gibbs distribution $\pi_{\beta}$ associated with Hamiltonian $H$ and inverse temperature $\beta$, the MH algorithm is a discrete-time Markov chain with transition matrix given by $P_\beta = P_\beta(N,H) = (P_\beta(x,y))_{x,y \in \mathcal{X}}$, where
$$P_\beta(x,y) := \begin{cases} N(x,y) \min \left\{ 1,e^{\beta \left(H(x)-H(y)\right)} \right\} = N(x,y) e^{-\beta (H(y)-H(x))_+}, &\mbox{if } x \neq y; \\
1	- \sum_{z: z \neq x} P_\beta(x,z), & \mbox{if } x = y. \end{cases}$$

For $\alpha \in [0,1]$ and $\psi \in \Psi(\pi)$, we compute that, for $x \neq y$,
\begin{align}
	\overline{(P_\beta)}_\alpha(Q_\psi)(x,y) &= \alpha N(x,y) e^{-\beta (H(y)-H(x))_+} + (1-\alpha) N(\psi(x),\psi(y)) e^{-\beta (H(\psi(y))-H(\psi(x)))_+} \nonumber \\
	&= (\alpha N(x,y) + (1-\alpha) N(\psi(x),\psi(y))) e^{-\beta (H(y)-H(x))_+} \nonumber \\
	&= P_\beta(\alpha N + (1-\alpha) Q_\psi N Q_\psi,H)(x,y), \label{eq:Pbetaalpha}
\end{align}
where the second equality follows from $H(x) = H(\psi(x))$. Therefore, the family $(	\overline{(P_\beta)}_\alpha(Q_\psi))_{\alpha \in [0,1]}$ can be interpreted as MH chains with modified proposals $(\alpha N + (1-\alpha) Q_\psi N Q_\psi)_{\alpha \in [0,1]}$ targeting the same $H$ at the same inverse temperature $\beta$. This interpretation also illustrates it is perhaps advantageous to use $\overline{(P_\beta)}_\alpha(Q)$: the proposal $\alpha N + (1-\alpha) Q N Q$ when $\alpha > 0$ has at least as many connections as the original proposal $N$, that is, $\{(x,y);~ N(x,y) > 0\} \subseteq \{(x,y);~ (\alpha N + (1-\alpha)QNQ)(x,y) > 0\} $. In this sense, this technique to improve mixing is in the spirit of \cite{GH19}.

%\begin{definition}[The classical MH $M_1$]\label{def:M1}
%	Given a proposal continuous-time ergodic Markov chain with generator $Q$ and a Gibbs distribution $\pi_{\beta}$ associated with Hamiltonian $H$ and inverse temperature $\beta$, the MH algorithm $X^{M_1} = (X^{M_1}_t)_{t \geq 0}$ is a continuous-time Markov chain with generator given by $M_{1} = M_{1}(Q,H,\beta) = (M_{1}(x,y))_{x,y \in \mathcal{X}}$, where
%	$$M_{1}(x,y) := \begin{cases} Q(x,y) \min \left\{ 1,e^{\beta \left(H(x)-H(y)\right)} \right\} = Q(x,y) e^{-\beta (H(y)-H(x))_+}, &\mbox{if } x \neq y; \\
%		- \sum_{z: z \neq x} M_{1}(x,z), & \mbox{if } x = y. \end{cases}$$
%	We write $P^{M_1} = (P^{M_1}_{t})_{t \geq 0} = (P^{M_1}_{t}(x,y))_{x,y \in \mathcal{X},t \geq 0}$ to be the transition semigroup of $X^{M_1}$, where $P^{M_1}_{t}(x,y)$ is the transition probability of $X^{M_1}$ starting in state $x$ at time $0$ to state $y$ at time $t$.
%\end{definition}

To quantify the speed of convergence towards $\pi_\beta$ of the MH chain, we now recall an important parameter that is known as the hill-climbing constant, energy barrier or the critical height in the literature. In a broad sense, it measures the difficulty of navigating on the landscape of $H$. Precisely, we say that a path from $x$ to $y$ is any sequence of points starting from $x_0 = x, x_1, x_2,\ldots, x_n =y$ such that $N(x_{i-1},x_i) > 0$ for $i \in \llbracket n \rrbracket$. As $N$ is irreducible, for any $x \neq y$ such path exists. We write $\Gamma^{x,y} = \Gamma^{x,y}(N)$ to be the set of paths from $x$ to $y$, and elements of $\Gamma^{x,y}$ are denoted by $\gamma = (\gamma_i)_{i=0}^n$. The value of the Hamiltonian $H(x)$ can be interpreted as the elevation at $x$, and the highest elevation along a path $\gamma \in \Gamma^{x,y}$ is 
$$\mathrm{Elev}(\gamma) = \max\{H(\gamma_i);~\gamma_i \in \gamma\},$$ 
and the lowest possible highest elevation along path(s) from $x$ to $y$ is 
$$\mathbf{H}(x,y) := \min\{\mathrm{Elev}(\gamma);~\gamma \in \Gamma^{x,y}\}.$$
For $P_\beta(N,H)$, the associated critical height is defined to be
\begin{align}
	h(P_\beta) &:= \max_{x,y \in \mathcal{X}}\{\mathbf{H}(x,y) - H(x) - H(y)\} + \min_z H(z), \label{eq:cm1} 
\end{align} 

Using a classical result of \cite[Lemma 2.3, 2.7]{HS88}, the right spectral gap of the MH chain can be bounded using $h(P_\beta)$: there exists constants $0 < c = c(N) \leq C = C(N) < \infty$ that do not depend on $\beta$ such that
\begin{align}\label{eq:HSspectral}
	c e^{-\beta h(P_\beta)} \leq \gamma(P_\beta) \leq C e^{- \beta h(P_\beta)}.
\end{align}

We now compare the critical heights of the family $(\overline{(P_\beta)}_\alpha(Q))_{\alpha \in [0,1]}$, and demonstrates the optimality of $\alpha = 1/2$: the sampler $(1/2)(P_\beta + QP_\beta Q)$ has the smallest critical height within this family, thus leading to improved convergence over the original $P_\beta$. In addition to critical height, $(1/2)(P_\beta + QP_\beta Q)$ also enjoys some advantageous properties over the original $P$ in terms of entropic and spectral parameters as well as asymptotic variances if we recall Section \ref{sec:compare}. It justifies the decision to focus on analyzing $(1/2)(P_\beta + QP_\beta Q)$ in this context.

\begin{proposition}\label{prop:criticalh}
	Let $P_\beta$ be the transition matrix of the MH chain with ergodic proposal $N$ and target distribution $\pi_\beta$, and $Q \in \mathcal{I}(\pi_\beta) \cap \mathcal{L}$. We have
	\begin{itemize}
		\item(Similarity preserves critical height) \begin{align*}
			h(P_\beta) &= h(Q P_\beta Q).
		\end{align*}
	
		\item(Optimality of $\alpha = 1/2$) For $\alpha \in [0,1]$, 
		\begin{align*}
			h(\overline{(P_\beta)}_\alpha(Q)) \leq h(P_\beta).
		\end{align*}
		In particular, 
		\begin{align*}
			\min_{\alpha \in [0,1]} h(\overline{(P_\beta)}_\alpha(Q)) = h(\overline{(P_\beta)}_{1/2}(Q)).
		\end{align*}
	\end{itemize}
\end{proposition}

The above result also illustrates a way to tune $Q$: we should seek to choose $Q$ such that the critical height $h(\overline{(P_\beta)}_{1/2}(Q))$ is minimized. In the remaining of this section, we shall specialize into a discrete bimodal example on a line that has been investigated in the literature \cite{MZ03}.

Specifically, we consider $\mathcal{X} = \{-J, -J+1,\ldots,J-1,J\}$ for $J \in \mathbb{N}$, $H(x) = -|x|$ for $x < J-1$, $H(J-1) = -J$ and $H(J) = - J - 1$, while the proposal chain $N$ is taken to be a simple nearest-neighbor random walk on $\mathcal{X}$ with holding probability of $1/2$ at the two boundaries $\pm J$. In this setting, there is a global mode of $\pi_\beta$ at $J$ and a local mode located at $- J$. It can readily be seen that
\begin{align*}
	h(P_\beta) = J = H(0) - H(-J).
\end{align*}
The bottleneck of mixing in this case is the hill at $x = 0$ that separates the two modes, in which the MH chain needs to climb over if it is initiated at $x < 0$ or $x > 0$, which is exponentially unlikely as $\beta \to \infty$.

Let $\psi(-J) = J-1, \psi(J-1) = -J$ and $\psi(x) = x$ for $x \in \mathcal{X}\backslash\{-J,J-1\}$, and we take $Q = Q_\psi$. It can readily be seen that $Q^2 = I$ and $Q^* = Q$ since $H(-J) = H(J-1) = -J$. It turns out this choice of $Q$ is optimal since the critical height of $\overline{(P_\beta)}_{1/2}(Q)$ is zero. To see that, since $h(\overline{(P_\beta)}_{1/2}(Q))$ is attained in a path from a local minimum to a global minimum of $H$, the lowest elevation is zero since $QNQ(-J,J) = N(J-1,J) = 1/2 > 0$. This is graphically illustrated in Figure \ref{fig:criticalheight}.
\begin{figure}[H]
	\centering
	\includegraphics[width=0.75\textwidth]{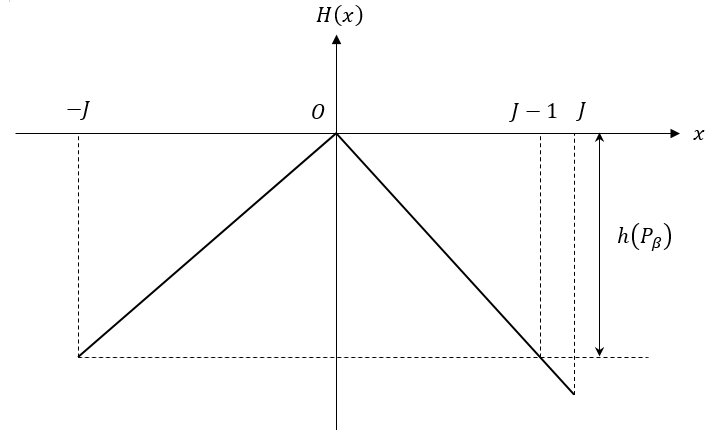}
	\caption{A landscape with a local minimum at $-J$ and a global minimum at $J$. The original critical height is $h(P_\beta) = J$, while $h(\overline{(P_\beta)}_{1/2}(Q)) = 0$. The reason is that there is a zero elevation path from $-J$ to $J-1$ to $J$ owing to the choice of $Q$.}
	\label{fig:criticalheight}
\end{figure}

Using \eqref{eq:HSspectral}, the relaxation time of $\overline{(P_\beta)}_{1/2}(Q)$ is subexponential while that of $P_\beta$ is exponential in $\beta$ and $J$. Note that at each time there is at most two additional permutation steps in $(1/2)(P_\beta + Q P_\beta Q)$ compared with the original $P_\beta$.

\begin{proposition}
	In the bimodal example, we have
	\begin{align*}
		\lim_{\beta \to \infty} \dfrac{1}{\beta} \ln (t_{rel}(\overline{(P_\beta)}_{1/2}(Q))) = 0, \\
		\lim_{\beta \to \infty} \dfrac{1}{\beta} \ln (t_{rel}(P_\beta)) = J.
	\end{align*}
\end{proposition}

In fact, a finer polynomial upper bound can be obtained for $t_{rel}(\overline{(P_\beta)}_{1/2}(Q))$. Precisely, in the notations of \cite{I94}, we have
\begin{align*}
	b_{\Gamma} \leq \dfrac{2J(2J-1)}{2} = 2J^2 - J, \quad \gamma_{\Gamma} \leq 2(2J-1) + 4 = 4J + 2, \quad d^* = 4.
\end{align*}
Applying \cite[Theorem $4.1$]{I94} in view of these upper bounds leads to
\begin{proposition}
	In the bimodal example, we have
	\begin{align*}
		t_{rel}(\overline{(P_\beta)}_{1/2}(Q)) \leq b_\Gamma \gamma_\Gamma d^* = \mathcal{O}(J^3).
	\end{align*}
\end{proposition}

\section{Improving discrete-uniform samplers with general permutations and projections}\label{sec:general}

In this section, we shall consider ergodic $P \in \mathcal{S}(\pi)$ where $\pi$ is the discrete uniform distribution on $\mathcal{X}$. In earlier sections, we have restricted ourselves to $Q \in \mathcal{I}(\pi) \cap \mathcal{L}$, the set of isometric involution transition matrices, and shown that such $Q$s are induced by equi-probability permutations in Proposition \ref{prop:characterizIpi}. In this paper, we shall relax this assumption of $Q$ to general permutation matrix in this section only.

In this setting, it is not necessary to use MCMC samplers to sample from the discrete uniform: if we have a $Q_\sigma$ drawn uniformly at random from the set of permutation matrices, then we can consider $Q_\sigma e_1$, where $e_1$ is the vector with $1$ in the first entry and $0$ in all remaining entries. This resulting vector would have a $1$ in a position sampled uniformly from $\llbracket n \rrbracket$. The main message in this section is that the kernel $\overline{P}(Q)$ has improved mixing time over the original $P$.

Precisely, we define
\begin{align*}
	\mathcal{Q} := \{Q_\psi; \psi \in \mathbf{P}\}
\end{align*}
and let $Q \in \mathcal{Q}$. If $Q = Q_\psi$, we see that $Q^* = Q_{\psi^{-1}}$, the $\ell^2(\pi)$-adjoint of $Q$. We also note that in general $Q^* \neq Q$. In addition, it is obvious to see that $Q$ is unitary since $Q Q^* = Q^* Q = I$. Generalizing \eqref{def:overlineP} to $Q \in \mathcal{Q}$, we analogously define
\begin{align*}
	\overline{P}(Q) := \dfrac{1}{2}(P + QP^*Q).
\end{align*}
The advantage of working under the setting of discrete uniform $\pi$ is that all of $P, QP, PQ, QPQ, \overline{P}(Q) \in \mathcal{S}(\pi)$, thus it is sensible to compare these transition matrices as candidate samplers of $\pi$.

We first demonstrate that a few results in earlier sections such as Section \ref{sec:basicdef} and Section \ref{sec:compare} can be generalized to a general permutation $Q$. Our first result states that, in terms of one-step KL divergence from $\Pi$ or the KL-divergence Dobrushin coefficient, the samplers $P, QP, PQ, QPQ$ cannot be distinguished. The proof is omitted as it is analogous to Proposition \ref{prop:convergegeneral}.

\begin{proposition}\label{prop:convergegeneral}
	Let $\pi$ be the discrete uniform distribution, $P \in \mathcal{S}(\pi)$ and $Q \in \mathcal{Q}$ be a permutation matrix. Let $\Pi$ be the matrix where each row equals to $\pi$. We have
	\begin{itemize}
		\item(One-step contraction measured by $D^{\pi}_{KL}$) \begin{align*}%\label{eq:comparegeneral}
			D^{\pi}_{KL}(P \| \Pi) &= D^{\pi}_{KL}(P Q \| \Pi) = D^{\pi}_{KL}(Q P \| \Pi) = D^{\pi}_{KL}(Q P Q\| \Pi).
		\end{align*}
		
		\item(KL-divergence Dobrushin coefficient)
		\begin{align*}%\label{eq:compare3general}
			c_{KL}(P) = c_{KL}(PQ) = c_{KL}(QP) = c_{KL}(QPQ).
		\end{align*}
	\end{itemize}
	
\end{proposition}

The second result gives a Pythagorean identity under $D^{\pi}_{KL}$, and its proof is similar to Proposition \ref{prop:converge2}.

\begin{proposition}[Pythagorean identity]\label{prop:converge2general}
	Let $\pi$ be the discrete uniform distribution, $P \in \mathcal{S}(\pi)$ and $Q \in \mathcal{Q}$ be a permutation matrix. Let $\Pi$ be the matrix where each row equals to $\pi$. We have
	\begin{align*}
		D^{\pi}_{KL}(\overline{P}(Q) \| \Pi) &\leq  D^{\pi}_{KL}(P  \| \overline{P}(Q) ) +  D^{\pi}_{KL}(\overline{P}(Q) \| \Pi) = D^{\pi}_{KL}(P \| \Pi),
	\end{align*}
	and the equality holds if and only if $\overline{P}(Q) = P$. 
	
	Similarly, if $P$ is further assumed to be $\pi$-reversible, then
	\begin{align*}
		c_{KL}(\overline{P}(Q)) &\leq c_{KL}(P).
	\end{align*}
\end{proposition}

For $\alpha \in [0,1]$, we see that
\begin{align*}
	\overline{\alpha P + (1-\alpha) QP^*Q}(Q) = \overline{P}(Q).
\end{align*}
Together with Proposition \ref{prop:converge2general}, we observe that the choice of $\alpha = 1/2$ is optimal within the family $(\alpha P + (1-\alpha) QP^*Q)_{\alpha \in [0,1]}$ as it minimizes the KL divergence $D^{\pi}_{KL}$:

\begin{corollary}[Optimality of $\alpha = 1/2$]\label{cor:optimal1/2general}
	Let $\pi$ be the discrete uniform distribution, $P \in \mathcal{S}(\pi)$ and $Q \in \mathcal{Q}$ be a permutation matrix. We have
	\begin{align*}
		\min_{\alpha \in [0,1]} D^{\pi}_{KL}(\alpha P + (1-\alpha) QP^*Q \| \Pi) &= D^{\pi}_{KL}(\overline{P}(Q) \| \Pi).
	\end{align*}
\end{corollary}

To apply the projection sampler $\overline{P}(Q)$ in practice, one may seek to tune $Q$ using similar strategies discussed in Section \ref{sec:tuneQ}. For instance, using the Pythagorean identity in Proposition \ref{prop:converge2general}, we see that
\begin{align*}
	\argmin_{Q \in \mathcal{Q}} D^\pi_{KL}(\overline{P}(Q) \| \Pi) &= \argmax_{\psi \in \mathbf{P}} D^\pi_{KL}(P \| \overline{P}(Q_\psi) ),
\end{align*}
where the rightmost optimization problem is an assignment problem.

\subsection{The example of Diaconis-Holmes-Neal}

In this subsection, we specialize into the following: let $\mathcal{X} = \llbracket n \rrbracket$, and consider $P$ to be the nearest-neighbour simple random walk with holding probability of $1/2$ at the two endpoints $1$ and $n$, that is, $P(1,1) = P(n,n) = 1/2$, $P(x,x+1) = 1/2$ for $x \in \llbracket n-1 \rrbracket$ and $P(x,x-1) = 1/2$ for $x \in \llbracket 2,n \rrbracket$. Clearly $P$ satisfies the assumption of this section: it is ergodic and admits the discrete uniform stationary distribution. 

This $P$ demonstrates a diffusive behaviour in the sense that the underlying Markov chain has a worst-case total variation mixing time of the order of $n^2$. This motivates \cite{DHM00} to introduce a non-reversible lifting of $P$ that aims at correcting the diffusive behaviour, who also prove that order $n$ steps are necessary and sufficient for the lifted chain to mix in worst-case total variation distance.

Let us recall that for $\mu, \nu \in \mathcal{P}(\mathcal{X})$, the total variation distance between them is 
\begin{align*}
	\norm{\mu - \nu}_{TV} := \dfrac{1}{2} \sum_{x \in \mathcal{X}} |\mu(x) - \nu(x)|,
\end{align*}
and the worst-case total variation mixing time of the Markov chain associated with $P$ is, for $\varepsilon > 0$,
\begin{align*}
	t_{mix}(P,\varepsilon) := \inf \bigg\{n \in \mathbb{N}; \max_{x \in \mathcal{X}} \norm{P^n(x,\cdot) - \pi}_{TV} < \varepsilon \bigg\}.
\end{align*}

The main result of this subsection is that, if the permutation matrix $Q$ is drawn uniformly at random from $\mathcal{Q}$, then with high probability $t_{mix}(\overline{P}(Q),\varepsilon)$ is at most of the order $\ln n$:

\begin{proposition}
	Let $P \in \mathcal{S}(\pi)$ be the nearest-neighbour simple random walk on $\llbracket n \rrbracket$ described at the beginning of this subsection, and $Q$ be a permutation matrix drawn uniformly at random from $\mathcal{Q}$. For all $\varepsilon \in (0,1)$, there exists $C(\varepsilon)$ such that with high probability
	\begin{align*}
		t_{mix}(\overline{P}(Q),\varepsilon) \leq \dfrac{\ln n}{\ln 2} + C(\varepsilon) \sqrt{\ln n}.
	\end{align*}
\end{proposition}

\begin{proof}
	We apply \cite[Theorem $1.1$]{D24}, where we note that $P$ is $\pi$-reversible, and the entropy rate of $P$ is $\ln 2$.
\end{proof}

\section{Numerical experiments}\label{sec:numerical}

To demonstrate the speedup enjoyed by the projection samplers, in this section we present numerical results concerning Ising model on the line, Edwards-Anderson (EA) spin glass, and the Blume-Capel (BC) model. For reproducibility, the code used in our experiments is available at \url{https://github.com/mchchoi/permutation/tree/main}.

\textcolor{black}{Detailed descriptions of the experimental protocol including models, samplers, parameters and convergence diagnostics can be found in Appendix in Section \ref{subsec:protocol}.}

\subsection{Ising model on the line}

%The parameters of the experiments are described as follows. The target inverse temperature is $\beta = 2$, and the dimensionality is $d = 50$. For the exploratory $Q$ projection sampler, the exploratory chain is run at a high temperature with $\beta_e = 0.1$ for $100,000$ steps. The samplers are simulated for a total of $100,000$ steps.

The results are presented in Figure \ref{fig:9panel_Ising} \textcolor{black}{and Table \ref{tab:Ising}}. In these simulations, we see that all projection samplers are able to transverse or hop between the modes $\mathbf{+1}$ and $\mathbf{-1}$ using the permutations. \textcolor{black}{This improved mixing is also reflected in a larger average jump distance, closer sample mean of magnetization to $0$ and shorter CI for the fixed Q and adaptive projection samplers.} On the other hand, the standard MH chain seems to struggle to move between these two modes and exhibits a diffusive behaviour in the magnetization trajectories. \textcolor{black}{We also note that both samplers reach similar Hamiltonian function values.}

\begin{table}[H]
	\centering
	\scriptsize  % or \scriptsize
	\begin{tabular}{l|cc|cc|cc}
		\toprule
		& \multicolumn{2}{c|}{Fixed $Q$} & \multicolumn{2}{c|}{Adaptive} & \multicolumn{2}{c}{Exploratory} \\
		& Standard MH & Fixed $Q$ & Standard MH & Adaptive & Standard MH & Exploratory \\
		\midrule
		Sample mean             & -0.24 & 0.10 & -0.27 & 0.04 & 0.38 & -0.07 \\
		95\% CI                 & [-0.76,0.28] & [-0.14,0.35] & [-0.92,0.38] & [-0.31,0.40] & [-0.10,0.87] & [-0.64,0.50] \\
		Avg. jump distance      & 0.00168 & 0.00235 & 0.00158 & 0.00408 & 0.00144 & 0.00164 \\
		\bottomrule
	\end{tabular}
	\caption{Summary statistics of magnetization for different samplers in Ising model on the line.}
	\label{tab:Ising}
\end{table}

\begin{figure}[H]
	\centering
	  % --- Row 1 ---
	\begin{subfigure}[t]{\textwidth}
		\centering
		\begin{subfigure}[t]{0.3\textwidth}
			\includegraphics[width=\linewidth]{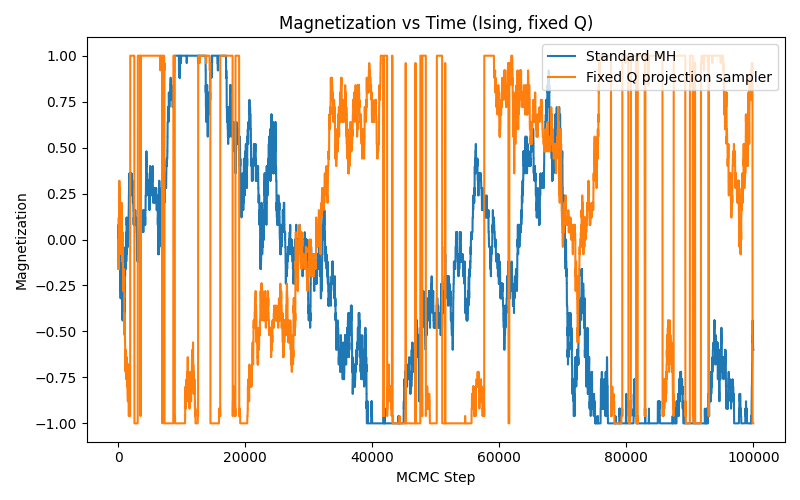}
		\end{subfigure}
		\hfill
		\begin{subfigure}[t]{0.3\textwidth}
			\includegraphics[width=\linewidth]{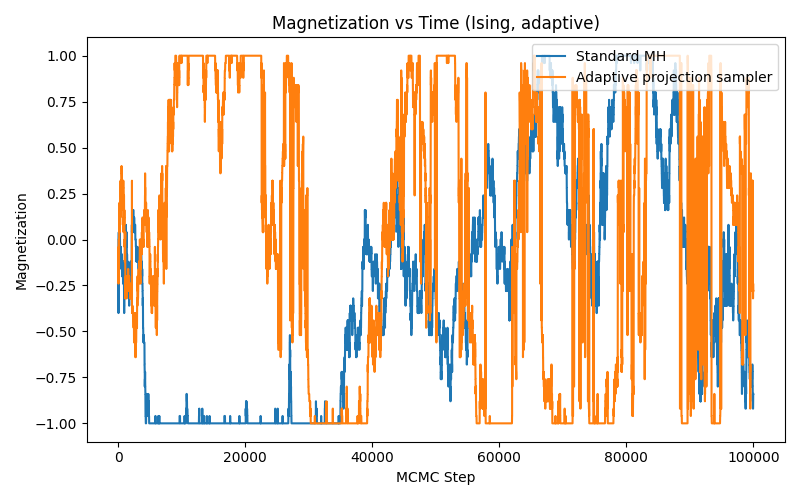}
		\end{subfigure}
		\hfill
		\begin{subfigure}[t]{0.3\textwidth}
			\includegraphics[width=\linewidth]{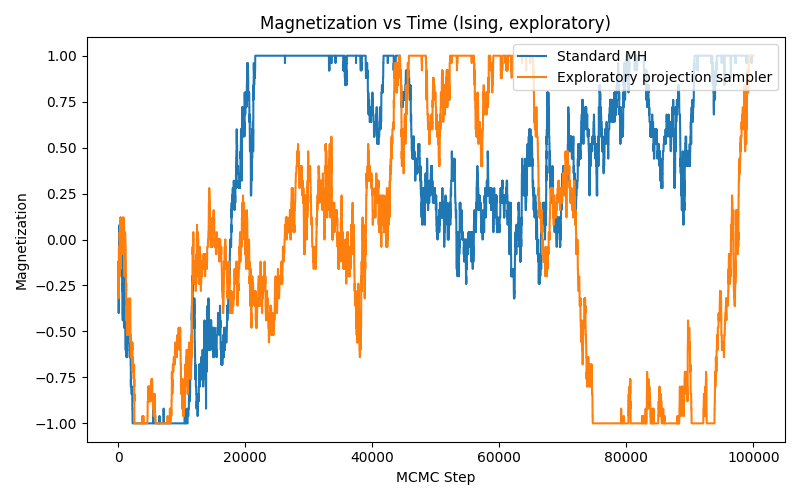}
		\end{subfigure}
		\caption{Magnetization against time.}
	\end{subfigure}
	
	\vspace{1em}
	
	\begin{subfigure}[t]{\textwidth}
		\centering
		\begin{subfigure}[t]{0.3\textwidth}
			\includegraphics[width=\linewidth]{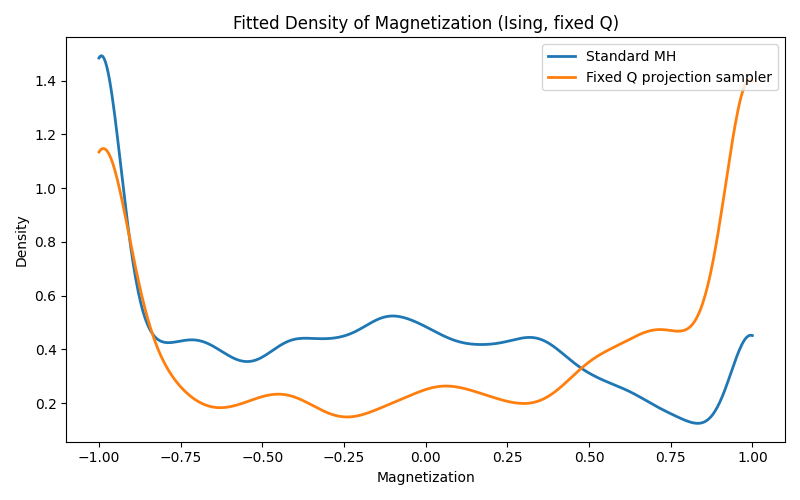}
		\end{subfigure}
		\hfill
		\begin{subfigure}[t]{0.3\textwidth}
			\includegraphics[width=\linewidth]{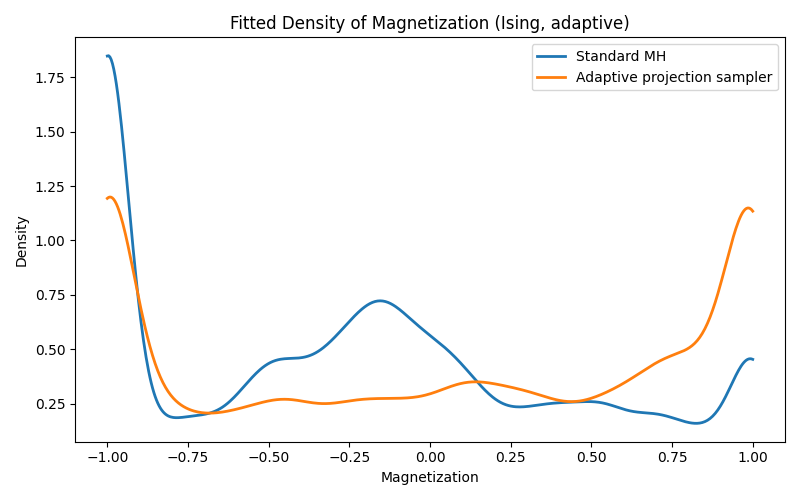}
		\end{subfigure}
		\hfill
		\begin{subfigure}[t]{0.3\textwidth}
			\includegraphics[width=\linewidth]{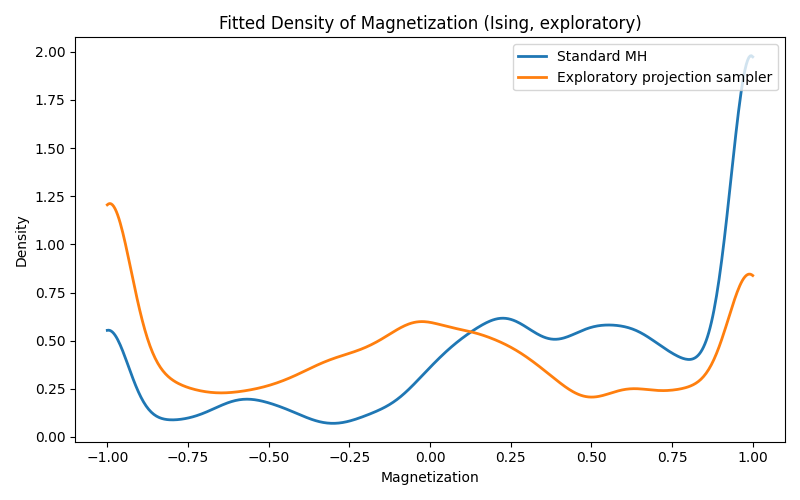}
		\end{subfigure}
		\caption{\textcolor{black}{Fitted density of magnetization.}}
	\end{subfigure}
	
	\vspace{1em}
	
	  \begin{subfigure}[t]{\textwidth}
		\centering
		\begin{subfigure}[t]{0.3\textwidth}
			\includegraphics[width=\linewidth]{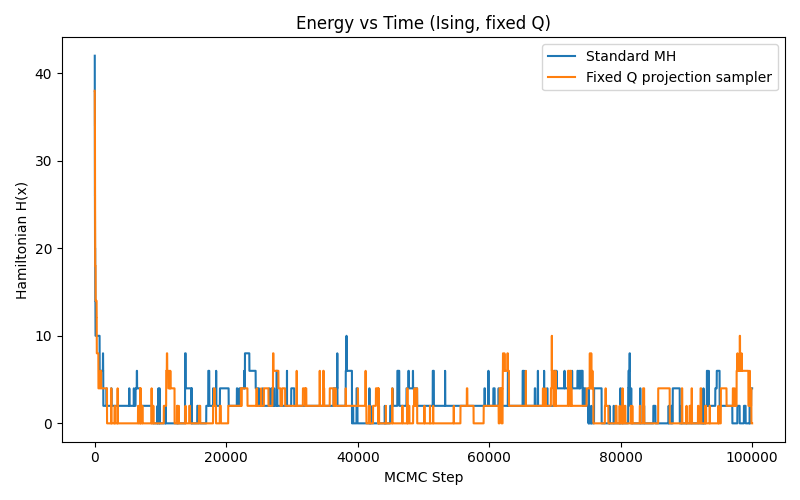}
		\end{subfigure}
		\hfill
		\begin{subfigure}[t]{0.3\textwidth}
			\includegraphics[width=\linewidth]{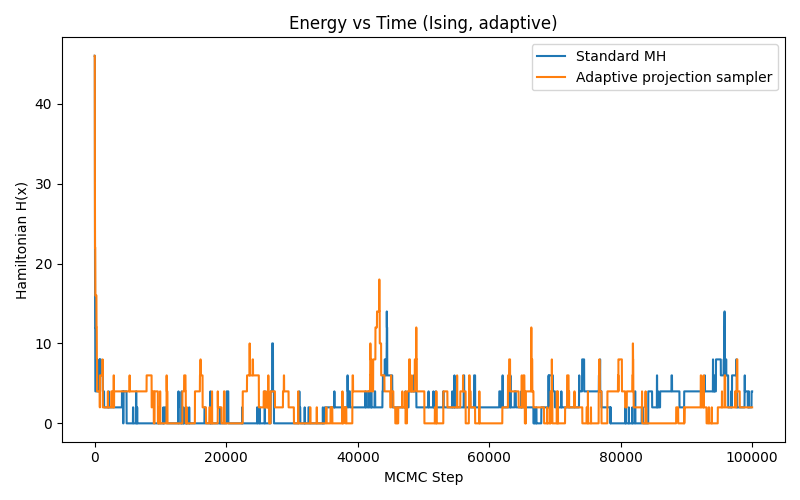}
		\end{subfigure}
		\hfill
		\begin{subfigure}[t]{0.3\textwidth}
			\includegraphics[width=\linewidth]{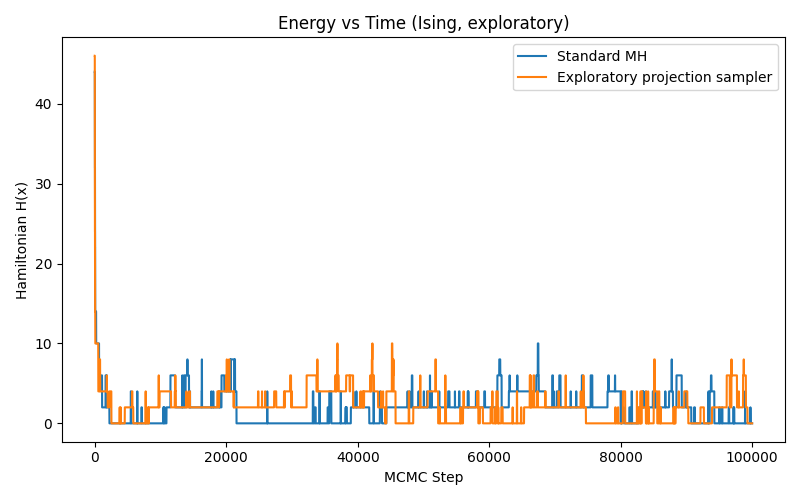}
		\end{subfigure}
		\caption{Hamiltonian function value against time.}
	\end{subfigure}
	
	\caption{Ising model on the line. From left to right are the standard MH (blue) compared with fixed $Q$, adaptive and exploratory projection samplers (orange).}
	\label{fig:9panel_Ising}
\end{figure}

\subsection{EA spin glass}

%The parameters of the experiments are described as follows. The target inverse temperature is $\beta = 2$, and the dimensionality is $d = 50$. For the exploratory $Q$ projection sampler, the exploratory chain is run at a high temperature with $\beta_e = 0.1$ for $100,000$ steps. For each set of experiments, the samplers are given the same set of realizations of $(J_{i,j})$. The samplers are simulated for a total of $100,000$ steps.

The results are presented in Figure \ref{fig:9panel_EA} \textcolor{black}{and Table \ref{tab:EA}}. In these simulations, we see that all projection samplers are able to reach a lower Hamiltonian function value than the standard MH. Notably, the adaptive and the exploratory projection sampler are able to take advantage of the permutations \textcolor{black}{as evident in the magnetization against time plot.} \textcolor{black}{These two samplers also have a larger average jump distance than that of standard MH.}

\begin{table}[H]
	\centering
	\scriptsize  % or \scriptsize
	\begin{tabular}{l|cc|cc|cc}
		\toprule
		& \multicolumn{2}{c|}{Fixed $Q$} & \multicolumn{2}{c|}{Adaptive} & \multicolumn{2}{c}{Exploratory} \\
		& Standard MH & Fixed $Q$ & Standard MH & Adaptive & Standard MH & Exploratory \\
		\midrule
		Sample mean             & 0.12 & 0.12 & 0.07 & -0.01 &-0.08 & -0.08 \\
		95\% CI                 & [0.11,0.12] & [0.12,0.12] & [0.05,0.08] & [-0.03,0.02] & [-0.08,-0.08] & [-0.08,-0.08] \\
		Avg. jump distance      & 0.00012 & 0.00006 & 0.00010 & 0.00015 & 0.00003 & 0.00010 \\
		\bottomrule
	\end{tabular}
	\caption{Summary statistics of magnetization for different samplers in EA spin glass.}
	\label{tab:EA}
\end{table}

\begin{figure}[H]
	\centering
	% --- Row 1 ---
	\begin{subfigure}[t]{\textwidth}
		\centering
		\begin{subfigure}[t]{0.3\textwidth}
			\includegraphics[width=\linewidth]{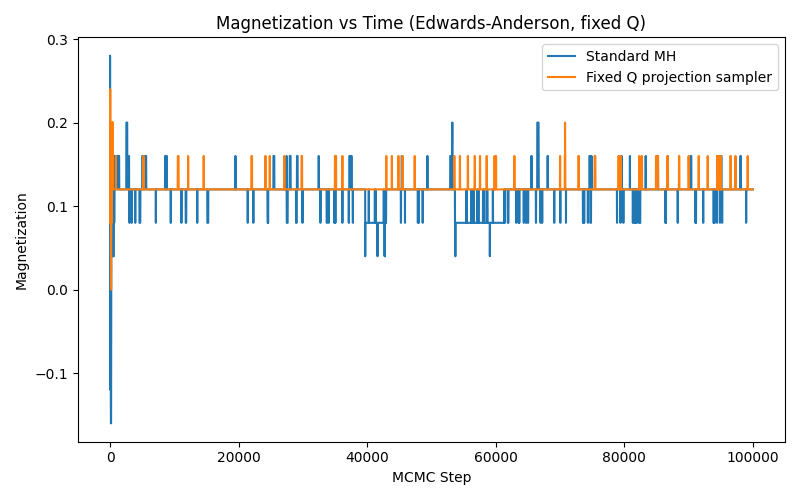}
		\end{subfigure}
		\hfill
		\begin{subfigure}[t]{0.3\textwidth}
			\includegraphics[width=\linewidth]{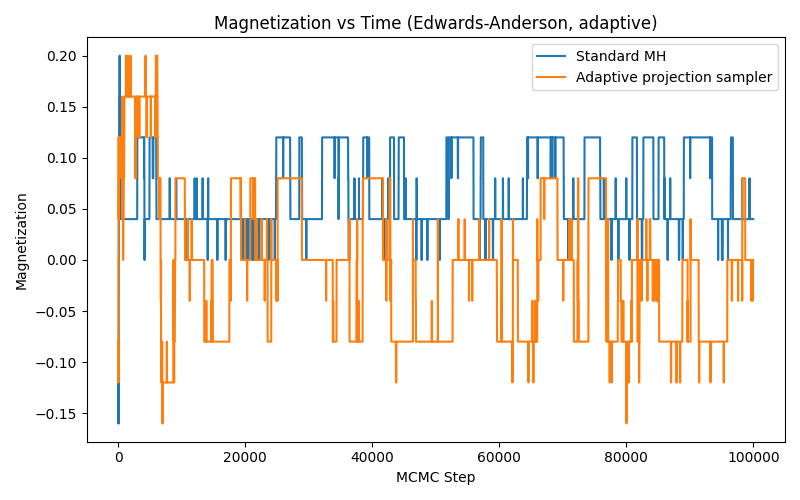}
		\end{subfigure}
		\hfill
		\begin{subfigure}[t]{0.3\textwidth}
			\includegraphics[width=\linewidth]{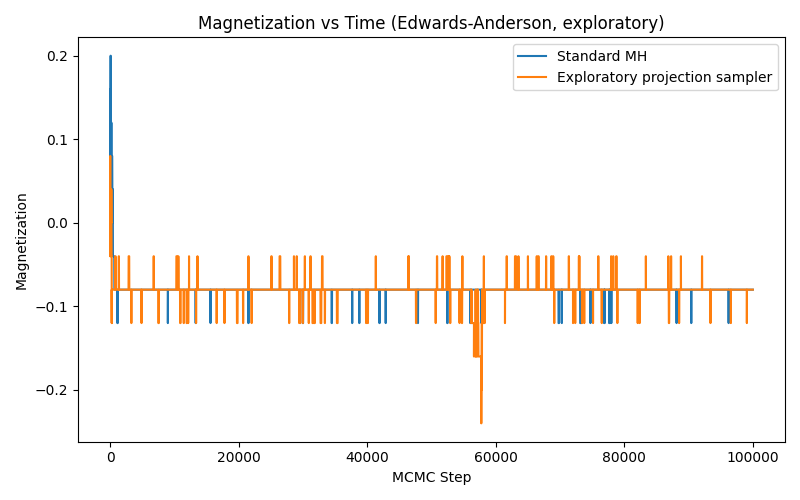}
		\end{subfigure}
		\caption{Magnetization against time.}
	\end{subfigure}
	
	\vspace{1em}
	
	\begin{subfigure}[t]{\textwidth}
		\centering
		\begin{subfigure}[t]{0.3\textwidth}
			\includegraphics[width=\linewidth]{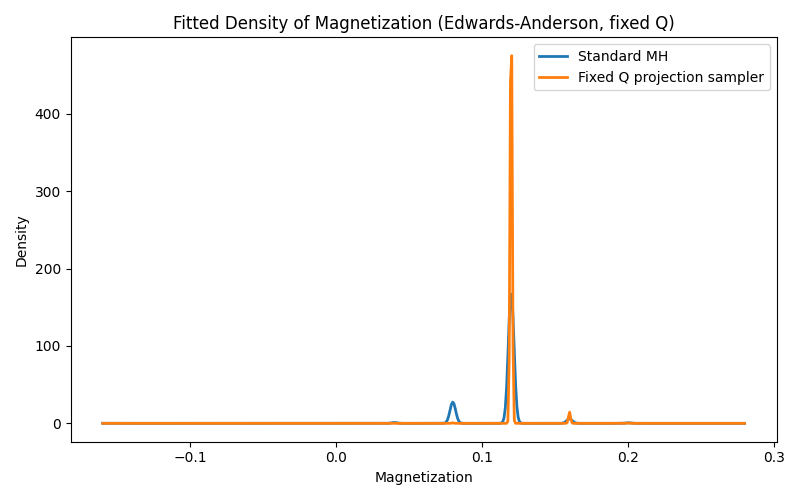}
		\end{subfigure}
		\hfill
		\begin{subfigure}[t]{0.3\textwidth}
			\includegraphics[width=\linewidth]{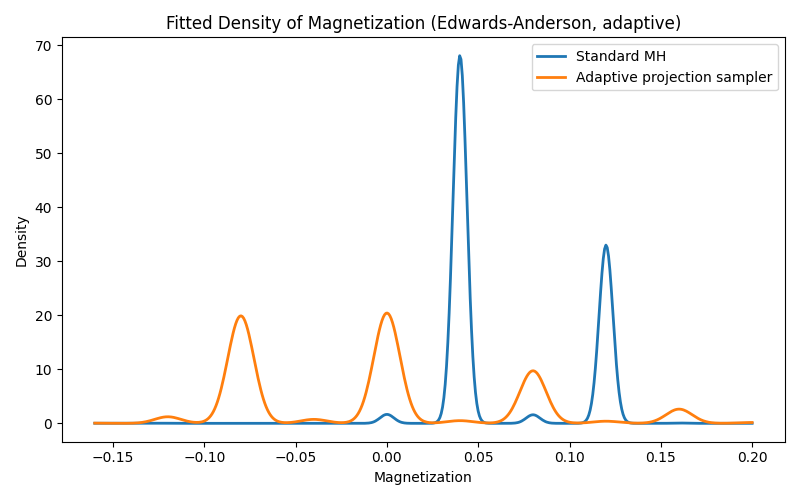}
		\end{subfigure}
		\hfill
		\begin{subfigure}[t]{0.3\textwidth}
			\includegraphics[width=\linewidth]{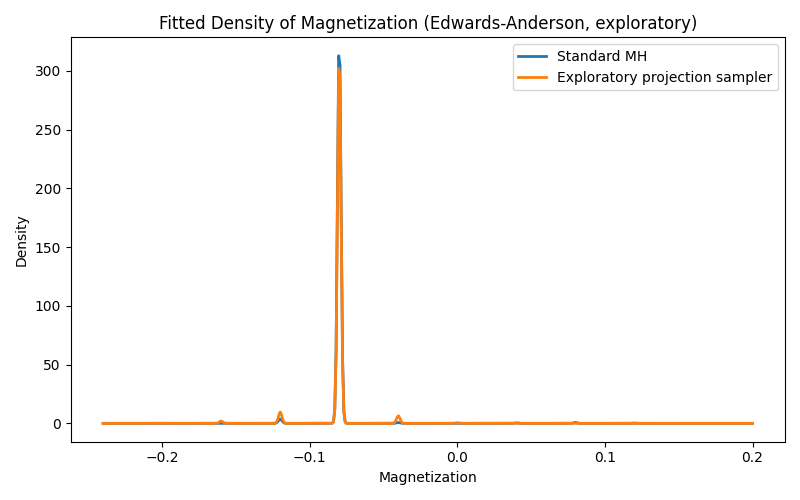}
		\end{subfigure}
		\caption{\textcolor{black}{Fitted density of magnetization.}}
	\end{subfigure}
	
	\vspace{1em}
	
	\begin{subfigure}[t]{\textwidth}
		\centering
		\begin{subfigure}[t]{0.3\textwidth}
			\includegraphics[width=\linewidth]{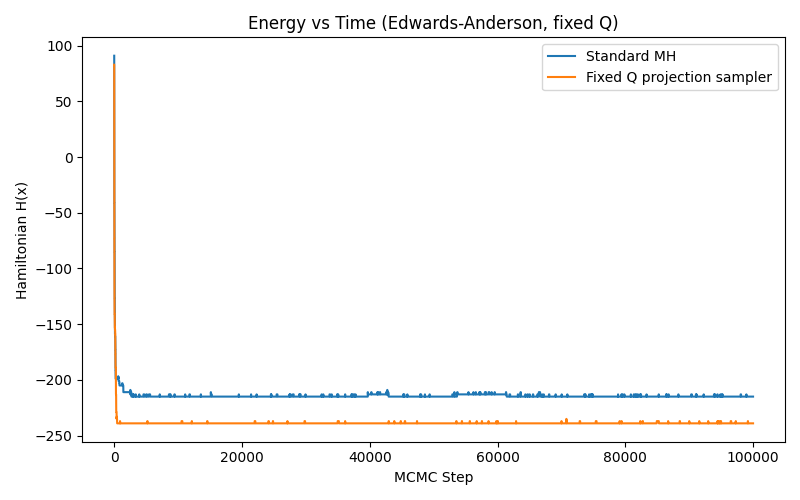}
		\end{subfigure}
		\hfill
		\begin{subfigure}[t]{0.3\textwidth}
			\includegraphics[width=\linewidth]{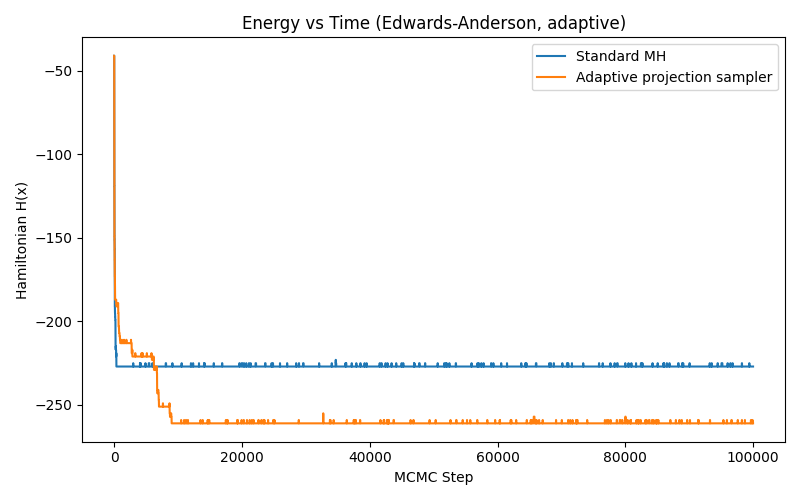}
		\end{subfigure}
		\hfill
		\begin{subfigure}[t]{0.3\textwidth}
			\includegraphics[width=\linewidth]{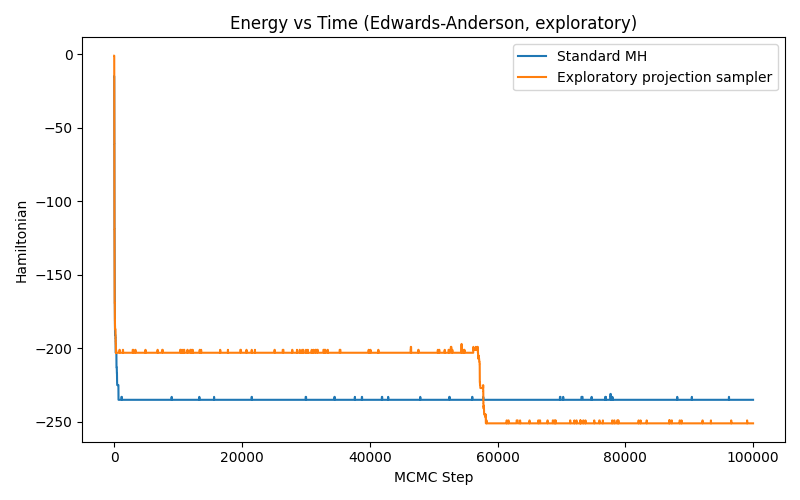}
		\end{subfigure}
		\caption{Hamiltonian function value against time.}
	\end{subfigure}
	
	\caption{EA model. From left to right are the standard MH (blue) compared with fixed $Q$, adaptive and exploratory projection samplers (orange).}
	\label{fig:9panel_EA}
\end{figure}

\subsection{BC model}

%The parameters of the experiments are described as follows. The target inverse temperature is $\beta = 3$, and the dimensionality is $d = 50$. For the exploratory $Q$ projection sampler, the exploratory chain is ran at a high temperature with $\beta_e = 0.1$ for $200,000$ steps. The samplers are simulated for a total of $200,000$ steps.

The results are presented in Figure \ref{fig:9panel_BC} \textcolor{black}{and Table \ref{tab:BC}}. In these simulations, we see that all projection samplers are able to identify the three modes, \textcolor{black}{particlarly $\mathbf{+1}, \mathbf{-1}$,} while the standard MH may not be able to locate these three. \textcolor{black}{From the table it also reveals improved mixing of projection samplers as evident by closer to $0$ sample mean magnetization in all projection samplers, larger average jump distance in fixed $Q$ and adaptive sampler, and a shorter CI for the adaptive sampler.} \textcolor{black}{We also note that the samplers reach similar Hamiltonian function values, thereby verifying that these are ``noisy descent" algorithms in low temperature.}

\begin{table}[H]
	\centering
	\scriptsize  % or \scriptsize
	\begin{tabular}{l|cc|cc|cc}
		\toprule
		& \multicolumn{2}{c|}{Fixed $Q$} & \multicolumn{2}{c|}{Adaptive} & \multicolumn{2}{c}{Exploratory} \\
		& Standard MH & Fixed $Q$ & Standard MH & Adaptive & Standard MH & Exploratory \\
		\midrule
		Sample mean             & -0.04 & 0.04 & -0.22 & -0.01 & 0.13 & 0.05 \\
		95\% CI                 & [-0.22,0.14] & [-0.26,0.33] & [-0.51,0.08] & [-0.17,0.14] & [-0.04,0.30] & [-0.24,0.34] \\
		Avg. jump distance      & 0.00142 & 0.00159 & 0.00142 & 0.00166 & 0.00144 & 0.00136 \\
		\bottomrule
	\end{tabular}
	\caption{Summary statistics of magnetization for different samplers in BC model.}
	\label{tab:BC}
\end{table}

\begin{figure}[H]
	\centering
	% --- Row 1 ---
	\begin{subfigure}[t]{\textwidth}
		\centering
		\begin{subfigure}[t]{0.3\textwidth}
			\includegraphics[width=\linewidth]{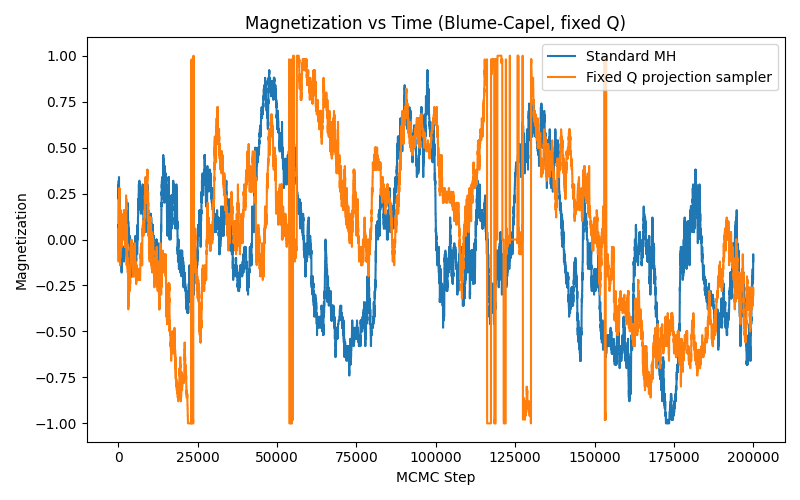}
		\end{subfigure}
		\hfill
		\begin{subfigure}[t]{0.3\textwidth}
			\includegraphics[width=\linewidth]{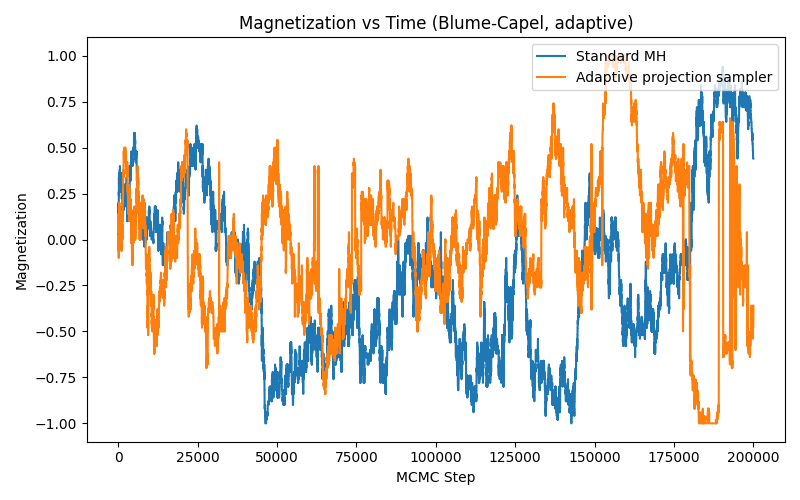}
		\end{subfigure}
		\hfill
		\begin{subfigure}[t]{0.3\textwidth}
			\includegraphics[width=\linewidth]{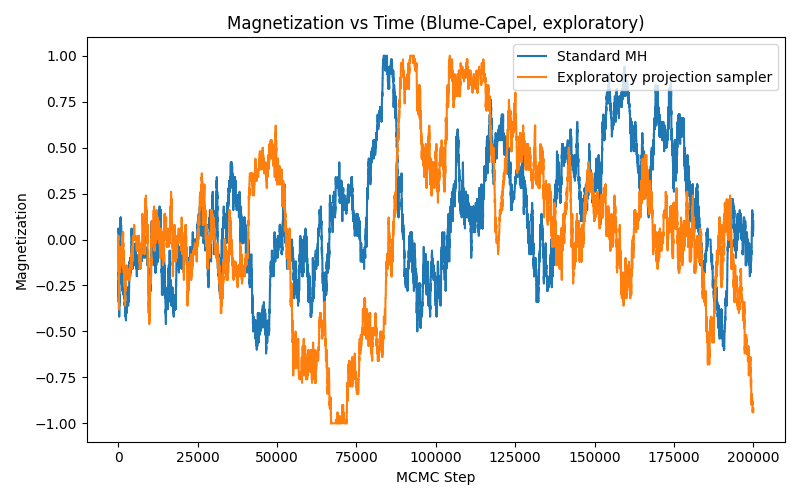}
		\end{subfigure}
		\caption{Magnetization against time.}
	\end{subfigure}
	
	\vspace{1em}
	
	\begin{subfigure}[t]{\textwidth}
		\centering
		\begin{subfigure}[t]{0.3\textwidth}
			\includegraphics[width=\linewidth]{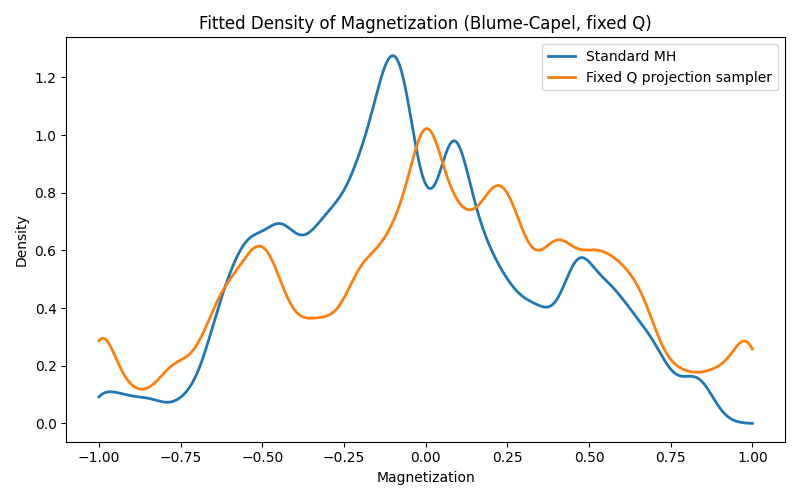}
		\end{subfigure}
		\hfill
		\begin{subfigure}[t]{0.3\textwidth}
			\includegraphics[width=\linewidth]{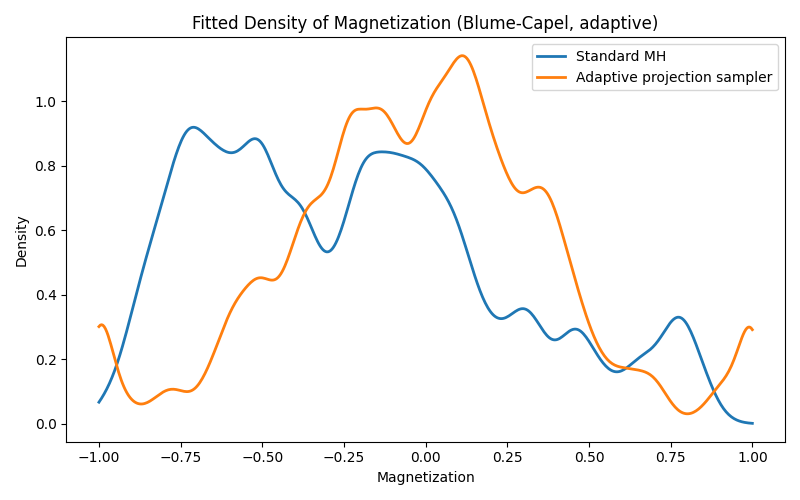}
		\end{subfigure}
		\hfill
		\begin{subfigure}[t]{0.3\textwidth}
			\includegraphics[width=\linewidth]{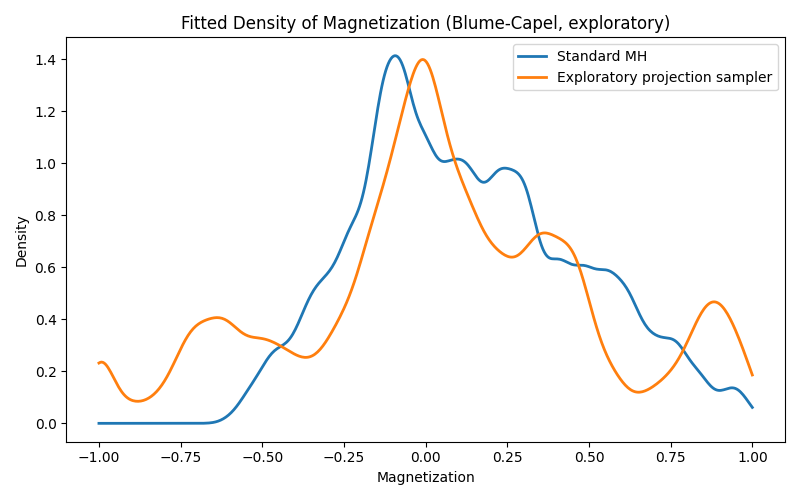}
		\end{subfigure}
		\caption{\textcolor{black}{Fitted density of magnetization.}}
	\end{subfigure}
	
	\vspace{1em}
	
	\begin{subfigure}[t]{\textwidth}
		\centering
		\begin{subfigure}[t]{0.3\textwidth}
			\includegraphics[width=\linewidth]{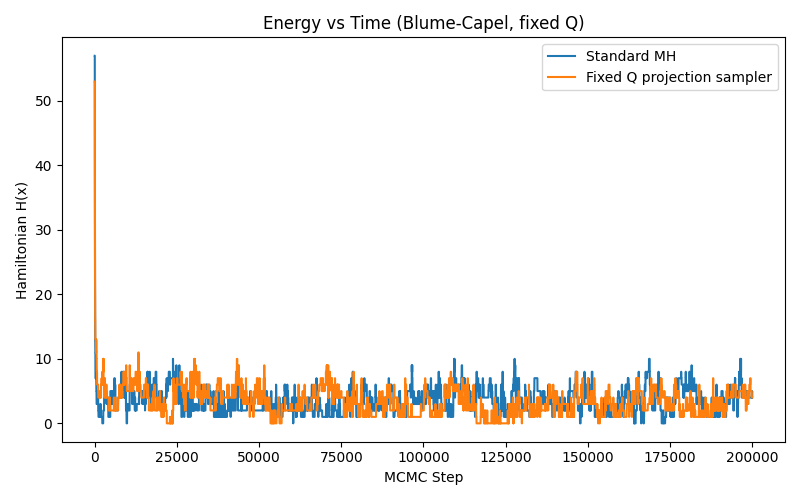}
		\end{subfigure}
		\hfill
		\begin{subfigure}[t]{0.3\textwidth}
			\includegraphics[width=\linewidth]{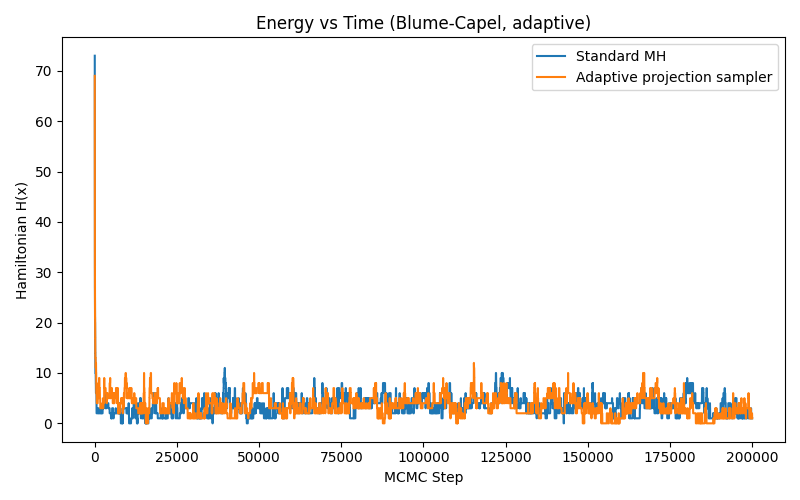}
		\end{subfigure}
		\hfill
		\begin{subfigure}[t]{0.3\textwidth}
			\includegraphics[width=\linewidth]{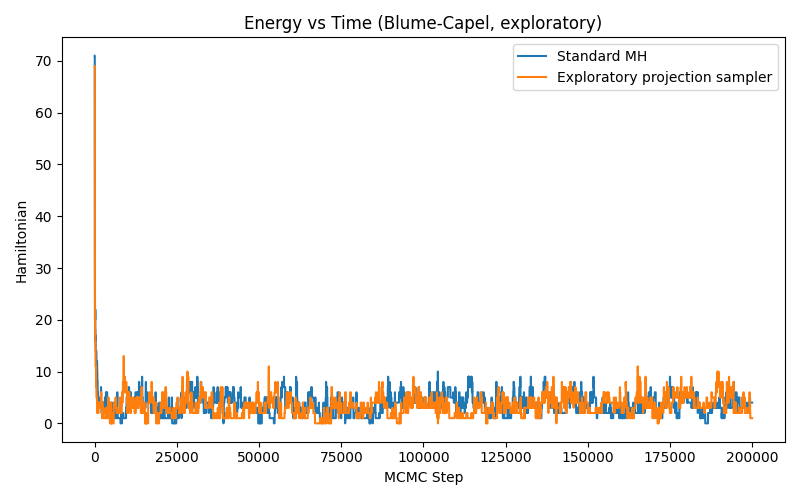}
		\end{subfigure}
		\caption{Hamiltonian function value against time.}
	\end{subfigure}
	
	\caption{BC model. From left to right are the standard MH (blue) compared with fixed $Q$, adaptive and exploratory projection samplers (orange).}
	\label{fig:9panel_BC}
\end{figure}

\section{Proofs of results}\label{sec:proof}

This section contains the proofs of various results that appear earlier in Section \ref{sec:basicdef}, \ref{sec:compare} and \ref{sec:MH} which involve standard techniques from linear algebra, probability and information theory.

\subsection{Proof of Proposition \ref{prop:deformKLprop}}

First, we prove non-negativity. 
\begin{align*}
	{}^Q D_{KL}(P \| L) = D^\pi_{KL}(QP \| QL)\geq 0,
\end{align*}
and the equality holds, by \cite[Proposition $3.1$]{WC23}, if and only if $QP = QL$ if and only if $P = L$. The proof for $D_{KL}^Q$ is similar and hence omitted.
	
Next, we prove duality. We see that,
\begin{align*}
	{}^Q D_{KL}(P \| L) = D^\pi_{KL}(QP \| QL) =  D^\pi_{KL}(P^* Q \| L^* Q) = D_{KL}^Q(P^* \| L^*),
\end{align*}
where the second equality follows from the bisection property \cite[Theorem $III.1$]{CW23}.

\subsection{Proof of Proposition \ref{prop:basic}}

	First, we prove \eqref{eq:QPyth}. It is easy to see that
	\begin{align*}
		{}^Q D_{KL}(P \| L) &= {}^Q D_{KL}(P \| \overline{P}) + \sum_x \pi(x) \sum_y QP(x,y) \ln \left(\dfrac{Q\overline{P}(x,y)}{QL(x,y)}\right),
	\end{align*}
	thus it suffices to show that the second term on the right hand side can be expressed as
	\begin{align}\label{eq:QPyth2}
		\sum_x \pi(x) \sum_y QP(x,y) \ln \left(\dfrac{Q\overline{P}(x,y)}{QL(x,y)}\right) = \sum_x \pi(x) \sum_y P^*Q(x,y) \ln \left(\dfrac{Q\overline{P}(x,y)}{QL(x,y)}\right).
	\end{align}
	To see \eqref{eq:QPyth2}, we compute that
	\begin{align*}
		\sum_x \sum_y \pi(x) QP(x,y) \ln \left(\dfrac{Q\overline{P}(x,y)}{QL(x,y)}\right) = \sum_x \sum_y \pi(y) P^* Q(y,x) \ln \left(\dfrac{Q\overline{P}(y,x)}{QL(y,x)}\right),
	\end{align*}
	where the equality uses the fact that $L, \overline{P} \in \mathcal{L}(\pi,Q)$.
	
	Next, we prove \eqref{eq:PythQ}. Applying \eqref{eq:QPyth} to $(P^*,L^*)$ and using the duality formula in Proposition \ref{prop:basic}, we note that
	\begin{align*}
		D_{KL}^Q(P \| L) &= {}^Q D_{KL}(P^* \| L^*) \\
		&= {}^Q D_{KL}(P^* \| \overline{P^*}) + {}^Q D_{KL}(\overline{P^*} \| L^*) \\
		&= D_{KL}^Q(P \| \overline{P}) +  D_{KL}^Q(\overline{P} \| L),
	\end{align*}
	where we use that $\overline{P^*}^* = \overline{P}$.

\subsection{Proof of Proposition \ref{prop:projectFrob}}

	\begin{align*}
		\norm{M - N}_F^2 &= \langle M - N, M - N \rangle_F \\
		&= \bigg\langle \dfrac{M - QMQ}{2} + \dfrac{M + QMQ}{2} - N, \dfrac{M - QMQ}{2} + \dfrac{M + QMQ}{2} - N \bigg\rangle_F \\
		&= \norm{M - \overline{M}(Q)}_F^2 + \norm{\overline{M}(Q) - N}_F^2 + 2 \bigg\langle \dfrac{M - QMQ}{2}, \dfrac{M + QMQ}{2} - N \bigg\rangle_F,
	\end{align*}
	and it suffices to show that the rightmost inner product equals to zero, that is,
	\begin{align*}
		\bigg\langle \dfrac{M - QMQ}{2}, \dfrac{M + QMQ}{2} - N\bigg\rangle_F = \mathrm{Tr}\left(\left(\dfrac{M - QMQ}{2}\right)^*\left(\dfrac{M + QMQ}{2} - N\right)\right) = 0.
	\end{align*}
	
	To see that, we shall prove that $\mathrm{Tr}(A) = \mathrm{Tr}(-A)$ with $A = \left(\frac{M - QMQ}{2}\right)^*\left(\frac{M + QMQ}{2} - N\right)$. We calculate that
	\begin{align*}
		\mathrm{Tr}\left(\left(\dfrac{M - QMQ}{2}\right)^*\left(\dfrac{M + QMQ}{2} - N\right)\right)
		&= \mathrm{Tr}\left(\left(\dfrac{M + QMQ}{2} - N\right)^*\left(\dfrac{M - QMQ}{2}\right)\right) \\
		&= \mathrm{Tr}\left(Q^2\left(\dfrac{M^* + QM^*Q}{2} - N^*\right)\left(\dfrac{M - QMQ}{2}\right)\right) \\
		&= \mathrm{Tr}\left(Q\left(\dfrac{M^* + QM^*Q}{2} - N^*\right)\left(\dfrac{M - QMQ}{2}\right)Q\right) \\ 
		&= \mathrm{Tr}\left(\left(\dfrac{QM^* + M^*Q}{2} - QN^*\right)\left(\dfrac{MQ - QM}{2}\right)\right) \\ 
		&= \mathrm{Tr}\left(\left(\dfrac{QM^*Q + M^*}{2} - N^*\right)\left(\dfrac{QMQ - M}{2}\right)\right) \\
		&= \mathrm{Tr}\left(\left(\dfrac{QMQ - M}{2}\right)^*\left(\dfrac{QMQ + M}{2} - N\right)\right),
	\end{align*}
	where the third equality follows from the cyclic property of the trace and the fifth equality makes use of $N^* = QN^*Q$. This completes the proof.

\subsection{Proof of Proposition \ref{prop:deformKLorigKL}}

	If $M = N$, then the equalities obviously hold and the values are all zeros. For $M \neq N$, we note that
	\begin{align*}
		D^\pi_{KL}(M \| N) &= D^\pi_{KL}((MQ)Q \| (NQ)Q) \\
		&\leq c_{KL}(Q) D^\pi_{KL}(MQ \| NQ) \\
		&\leq D^\pi_{KL}(MQ \| NQ) \\
		&\leq c_{KL}(Q) D^\pi_{KL}(M \| N) \leq D^\pi_{KL}(M \| N).
	\end{align*}
	The equalities hold and hence $D^\pi_{KL}(MQ \| NQ) = D^\pi_{KL}(M \| N)$.
	
	Using the duality in Proposition \ref{prop:deformKLprop} and the bisection property \cite[Theorem $III.1$]{CW23}, we note that
	\begin{align*}
		{}^Q D_{KL}(M \| N) = D^Q_{KL}(M^* \| N^*) = D^\pi_{KL}(M^* \| N^*) = D^\pi_{KL}(M \| N).
	\end{align*}

\subsection{Proof of Proposition \ref{prop:converge}}

\eqref{eq:compare} can readily be seen from \eqref{prop:deformKLorigKL}. We replace $P$ above in this proof by $QP$ to yield the rightmost equality of \eqref{eq:compare}.
	
Next, we prove \eqref{eq:compare3}. Using submultiplicativity of $c_{KL}$ \cite[Proposition $3.9$]{WC23}, we see that
$$c_{KL}(P) = c_{KL}((PQ)Q) \leq c_{KL}(PQ) c_{KL}(Q) \leq c_{KL}(PQ).$$
Replacing $P$ by $PQ$ in the equations above yields
$$c_{KL}(PQ) \leq c_{KL}(PQ^2) = c_{KL}(P).$$
	
Similarly, $c_{KL}(P) = c_{KL}(QP)$ can be shown. Replacing $P$ by $QP$ in the expressions above leads us to $c_{KL}(QP) = c_{KL}(QPQ)$.
	
Finally, we prove \eqref{eq:compare4}, which follows from the linearity and cyclic property of trace, $Q^2 = I$ and $\mathrm{Tr}(P) = \mathrm{Tr}(P^*)$.

\subsection{Proof of Proposition \ref{prop:converge2}}

	By taking $L = \Pi$ in Proposition \ref{prop:basic}, we see that 
	\begin{align*}
		D_{KL}^Q(P \| \Pi) &= D_{KL}^Q(P \| \overline{P}) +  D_{KL}^Q(\overline{P} \| \Pi)
	\end{align*}
	In view of Proposition \ref{prop:deformKLorigKL} and \ref{prop:converge}, we arrive at
	\begin{align*}
		D^{\pi}_{KL}(P \| \Pi) &= D_{KL}^Q(P \| \Pi) \\
		&= D_{KL}^Q(P \| \overline{P}) +  D_{KL}^Q(\overline{P} \| \Pi) \\
		&= D^{\pi}_{KL}(P Q \| \overline{P} Q) +  D^{\pi}_{KL}(\overline{P} Q \| \Pi) \\
		&=  D^{\pi}_{KL}(P  \| \overline{P} ) +  D^{\pi}_{KL}(\overline{P}  \| \Pi) \\
		&\geq D^{\pi}_{KL}(\overline{P} \| \Pi),
	\end{align*}
	where the equality holds if and only if $P = \overline{P}$ if and only if $P$ is itself $(\pi,Q)$-self-adjoint. 
	
	Using the convexity of $c_{KL}$ \cite[Proposition $3.9$]{WC23}, we see that
	\begin{align*}
		c_{KL}(\overline{P}) \leq \dfrac{1}{2}(c_{KL}(P) + c_{KL}(QPQ)) = c_{KL}(P),
	\end{align*}
	which the last equality follows from Proposition \ref{prop:converge}.

\subsection{Proof of Proposition \ref{prop:spectralspeedup}}

	This proposition is mainly a consequence of the Weyl's inequality \cite[Theorem $1.3$]{S94}. Specifically, in view of Proposition \ref{prop:similar}, we note that
	\begin{align*}
		\lambda_2(\overline{P}_{\alpha}(Q)) = \lambda_1(\overline{P}_{\alpha}(Q) - \Pi) \leq \lambda_1(\alpha(P - \Pi)) + \lambda_1((1-\alpha)(QPQ-\Pi)) = \lambda_2(P),
	\end{align*}
	where the equality holds in the Weyl's inequality if and only if there exists a common eigenvector $f$.
	
	Similarly, applying the Weyl's inequality again leads to
	\begin{align*}
		\lambda_n(\overline{P}_{\alpha}(Q)) \geq \lambda_n(\alpha P) + \lambda_n((1-\alpha)QPQ) = \lambda_n(P),
	\end{align*}
	where the equality holds in the Weyl's inequality if and only if there exists a common eigenvector $g$.
	
	In the positive-semi-definite setting, we note that by Weyl's inequality
	\begin{align*}
		\alpha \lambda_2(P) \leq \alpha \lambda_2(P) + (1-\alpha) \lambda_n(QPQ) &\leq \lambda_2(\overline{P}_{\alpha}(Q)), \\
		(1-\alpha) \lambda_2(QPQ) \leq (1-\alpha) \lambda_2(QPQ) + \alpha \lambda_n(P) &\leq \lambda_2(\overline{P}_{\alpha}(Q)),
	\end{align*}
	and the desired result follows from Proposition \ref{prop:similar}.

\subsection{Proof of Proposition \ref{prop:compareasympv}}

	First, we calculate that
	\begin{align*}
		4 \langle f,&g \rangle_\pi - 2 \langle (I-\overline{P}_\alpha(Q)) g, g \rangle_\pi - \langle f,f \rangle_\pi \\
		&= \alpha \left(4 \langle f,g \rangle_\pi - 2 \langle (I-P) g, g \rangle_\pi - \langle f,f \rangle_\pi\right) + (1-\alpha) \left(4 \langle f,g \rangle_\pi - 2 \langle (I-QPQ) g, g \rangle_\pi - \langle f,f \rangle_\pi\right).
	\end{align*}
	Taking the sup over $g \in \ell^2_0(\pi)$ and using \eqref{eq:asympvvar} leads to
	\begin{align*}
		v(f,\overline{P}_\alpha(Q)) &\leq \alpha v(f,P) + (1-\alpha) v(f,QPQ) \\
		&= \alpha v(f,P) + (1-\alpha) v(Qf,P),
	\end{align*}
	where the last equality follows from \eqref{eq:varfvarqf}.
	
	Finally, when $Qf = \pm f$ and $P$ is $\pi$-reversible, by replacing $P$ with $\overline{P}_\alpha(Q)$ above and recalling \eqref{eq:Pbarbar} earlier, we arrive at 
	\begin{align*}
		\min_{\alpha \in [0,1]} v(f,\overline{P}_\alpha(Q)) &= v(f,\overline{P}(Q)).
	\end{align*}

\subsection{Proof of Proposition \ref{prop:asympvarspeedup}}

	We first handle the results for worst-case asymptotic variance. Using \eqref{eq:worstasympvar}, Proposition \ref{prop:spectralspeedup} and the fact that the mapping $0 < c \mapsto \frac{1+c}{1-c}$ is strictly increasing, we have
	\begin{align*}
		V(\overline{P}_\alpha(Q)) = \dfrac{1+\lambda_2(\overline{P}_\alpha(Q))}{1-\lambda_2(\overline{P}_\alpha(Q))} \leq \dfrac{1 + \lambda_2(P)}{1-\lambda_2(P)} = V(P), 
	\end{align*}
	and so the equality holds if and only if $\lambda_2(\overline{P}_\alpha(Q)) = \lambda_2(P)$. If $P$ is positive-semi-definite, using the inequality that, for $a \in (0,1)$ and $c > 0$,
	$$a\dfrac{1+c}{1-c} \leq \dfrac{1+ac}{1-ac},$$
	we arrive at, by taking $a = \max\{\alpha,1-\alpha\}$,
	\begin{align*}
		\max\{\alpha,1-\alpha\} V(P) &\leq \dfrac{1+\max\{\alpha,1-\alpha\}\lambda_2(P)}{1-\max\{\alpha,1-\alpha\}\lambda_2(P)} \leq \dfrac{1+\lambda_2(\overline{P}_\alpha(Q))}{1-\lambda_2(\overline{P}_\alpha(Q))} = V(\overline{P}_\alpha(Q)),
	\end{align*}
	where the second inequality above follows from Proposition \ref{prop:spectralspeedup}.
	
	Next, we proceed to show the results for average-case asymptotic variance. In view of Proposition \ref{prop:compareasympv}, it suffices to show that
	\begin{align*}
		\overline{v}(P) = \int_{f \in \ell^2_0(\pi), \norm{f}_{\pi} = 1} v(Qf,P) d S(f),
	\end{align*}
	which is indeed true since
	\begin{align*}
		\int_{f \in \ell^2_0(\pi), \norm{f}_{\pi} = 1} v(Qf,P) d S(f) &= \int_{f \in \ell^2_0(\pi), \norm{f}_{\pi} = 1} v(f,QPQ) d S(f) \\
		&= \dfrac{2}{|\mathcal{X}|-1}\mathrm{Tr}(Z(QPQ)) - 1 \\
		&= \dfrac{2}{|\mathcal{X}|-1}\mathrm{Tr}(Z(P)) - 1 \\
		&= \overline{v}(P),
	\end{align*}
	where the first equality uses \eqref{eq:varfvarqf}, the second equality comes from \cite[Theorem $2.1$]{CCHP12}, the third equality uses $\lambda(P) = \lambda(QPQ)$, and the last equality uses again \cite[Theorem $2.1$]{CCHP12}.
	
	Finally, the optimality of $\alpha = 1/2$ can be seen by replacing $P$ with $\overline{P}_\alpha(Q)$ above and recalling \eqref{eq:Pbarbar} earlier.

\subsection{Proof of Proposition \ref{prop:criticalh}}

	We show the first item. Using Proposition \ref{prop:similar}, we have
	\begin{align*}
		\gamma(P_\beta) = \gamma(Q P_\beta Q).
	\end{align*}
	Now, if we consider \eqref{eq:HSspectral} and since $QP_\beta Q = P_\beta(QNQ,H)$, we have
	\begin{align*}
		h(P_\beta) &= -\lim_{\beta \to \infty} \dfrac{1}{\beta} \ln (\gamma(P_\beta)) \\
		&= -\lim_{\beta \to \infty} \dfrac{1}{\beta} \ln (\gamma(QP_\beta Q)) \\
		&= h(Q P_\beta Q).
	\end{align*}
	
	We proceed to prove the second item and it suffices to show for $\alpha \in (0,1)$ in view of the first item. By Proposition \ref{prop:spectralspeedup}, we see that
	\begin{align*}
		\gamma(\overline{(P_\beta)}_\alpha(Q)) \geq \gamma(P_\beta),
	\end{align*}
	and using \eqref{eq:HSspectral} and \eqref{eq:Pbetaalpha} again yield
	\begin{align*}
		h(P_\beta) &= -\lim_{\beta \to \infty} \dfrac{1}{\beta} \ln (\gamma(P_\beta)) \\
		&\geq -\lim_{\beta \to \infty} \dfrac{1}{\beta} \ln (\gamma(\overline{(P_\beta)}_\alpha(Q))) \\
		&= h(\overline{(P_\beta)}_\alpha(Q)).
	\end{align*}
	
	To demonstrate the optimality of $\alpha = 1/2$, we note that
	\begin{align*}
		\overline{(P_\beta)}_\alpha(Q) = P_\beta(\alpha N + (1-\alpha)QNQ,H),
	\end{align*}
	and applying the result above with $P_\beta$ replaced by $\overline{(P_\beta)}_\alpha(Q)$ leads to
	\begin{align*}
		h(\overline{(P_\beta)}_\alpha(Q)) \geq h(\overline{\overline{(P_\beta)}_\alpha(Q)}_{1/2}(Q)) = h(\overline{(P_\beta)}_{1/2}(Q)).
	\end{align*}

\section*{Acknowledgements}

 We thank Persi Diaconis for pointer to the work \cite{DKSC10} where alternating projections are used to analyze Gibbs sampler, and Zheyuan Lai for coding assistance. We are grateful to the two reviewers for careful reading and constructive comments. Michael Choi acknowledges the financial support of the project “MAPLE: Mechanistic Accelerated Prediction of Protein Secondary Structure via LangEvin Monte Carlo” with grant number 22-5715-P0001 under the NUS Faculty of Science Ministry of Education Tier 1 grant Data for Science and Science for Data collaborative scheme, project NUSREC-HPC-00001 and NUSREC-CLD-00001 for NUS HPC-AI Priority Projects for Research Program, as well as the startup funding of the National University of Singapore with grant number A-0000178-01-00. Max Hird is funded by an EPSRC DTP. Youjia Wang gratefully acknowledges the financial support from National University of Singapore via the Presidential Graduate Fellowship. 

\section{Appendix}\label{sec:appendix}

\subsection{Experimental protocol}\label{subsec:protocol}

{\color{black}In this subsection, we state clearly the experimental protocol for the experiments presented in Section \ref{sec:numerical}.
	
	\textbf{Models}.} We first give an overview of the models of interest:

\begin{itemize}
	\item Ising model on the line: the state space is $\mathcal{X} = \{-1,+1\}^d$, and the Hamiltonian is given by, for $\mathbf{x} = (x(1),\ldots,x(d)) \in \mathcal{X}$,
	\begin{align*}
		H(\mathbf{x}) = \sum_{i=1}^{d-1} (1 - x(i) x(i+1)).
	\end{align*}
	The global minima are $\mathbf{+1}$ and $\mathbf{-1}$, the all-ones and all-minus\textcolor{black}{-ones} configuration.
	
	\item EA spin glass: the state space is $\mathcal{X} = \{-1,+1\}^d$, and the Hamiltonian is given by, for $\mathbf{x} = (x(1),\ldots,x(d)) \in \mathcal{X}$,
	\begin{align*}
		H(\mathbf{x}) = -\sum_{i < j} J_{i,j} x(i) x(j),
	\end{align*}
	where $(J_{i.j})$ are i.i.d. coupling coefficients taking values on either $+1$ or $-1$ with equal probability. Unlike the Ising model on the line, the landscape of $H$ is random and the global minima may not be $\mathbf{+1}$ or $\mathbf{-1}$.
	
	\item BC model on the line (with zero chemical potential and zero external field): the state space is $\mathcal{X} = \{-1,0,+1\}^d$, and the Hamiltonian is given by, for $\mathbf{x} = (x(1),\ldots,x(d)) \in \mathcal{X}$,
	\begin{align*}
		H(\mathbf{x}) = \sum_{i=1}^{d-1} (x(i) - x(i+1))^2.
	\end{align*}
	The global minima are $\mathbf{+1}$, $\mathbf{0}$ and $\mathbf{-1}$.
\end{itemize}

{\color{black}\textbf{Samplers}.} We compare the following four samplers:

\begin{itemize}
	\item Standard Metropolis-Hastings (MH) $P_\beta(N,H)$: this serves as the baseline sampler for comparison. The proposal chain that we use is described as follows: one uniformly-at-random coordinate is selected out of the $d$ coordinates, and the spin of that coordinate is uniformly randomized to a value excluding the current spin. For instance, if the current spin is $-1$ in either the Ising or EA model, then it is flipped to $+1$. In the BC model, it is flipped to either $0$ or $+1$ with equal probability.
	
	\item Fixed $Q$ projection sampler $\overline{P_\beta}(Q)$: we leverage the symmetric structure of $\mathcal{X}$ in the above models to see that a natural $Q_\sigma$ in this context is given by $\sigma(\mathbf{x}) = -\mathbf{x}$ for all $\mathbf{x} \neq \mathbf{+1}, \mathbf{-1}$ and $\sigma(\mathbf{x}) = \mathbf{x}$ for $\mathbf{x} \in\{\mathbf{+1}, \mathbf{-1}\}$. In both the Ising and the EA model, we shall be using $\overline{P_\beta}(Q_\sigma)$.
	
	In the BC model, as there are three modes of the Gibbs distribution, we shall consider in addition the permutation $\psi(\mathbf{+1}) = \mathbf{0}$, $\psi(\mathbf{x}) = \mathbf{x}$ for $x \neq \mathbf{+1},\mathbf{0}$, and the resulting projection sampler
	\begin{align*}
		R_2 = \dfrac{1}{4}\left(P_\beta + Q_\sigma P_\beta Q_\sigma + Q_\psi P_\beta Q_\psi + Q_\psi Q_\sigma P_\beta Q_\sigma Q_\psi \right).
	\end{align*}
	
	\item Adaptive $Q$ projection sampler: Starting from the identity matrix as permutation matrix, this sampler adaptively updates the equi-probability permutation matrix on the fly in every 50 steps as described in Section \ref{subsec:tuneQsingle}.
	
	\item Exploratory $Q$ projection sampler: we first run a MH chain $P_{\beta_e}$ at a high temperature (i.e. small $\beta_e$) to explore the given landscape $H$. This results in a $Q$ as described in Section \ref{subsec:tuneQexplore}. Consequently, we use such $Q$ in the projection sampler $\overline{P_\beta}(Q)$ at the target $\beta$.
\end{itemize}

{\color{black}\textbf{Initial configuration}. For all experiments, the initial configuration of the samplers is drawn uniformly at random from the state space. In each experiment, we use the same initial state for both standard and projection samplers, to ensure fair comparison between these two.}

{\color{black}\textbf{Parameters and number of MCMC steps}. The parameters of the experiments are described as follows. 
	\begin{itemize}
		\item Ising model on the line: the target inverse temperature is $\beta = 2$, and the dimensionality is $d = 50$. For the adaptive sampler, the permutation matrix $Q$ is adaptively updated in every $50$ steps. For the exploratory projection sampler, the exploratory chain is run at a high temperature with $\beta_e = 0.1$ for $100,000$ steps. The samplers are simulated for a total of $100,000$ steps.
		
		\item EA spin glass: the target inverse temperature is $\beta = 2$, and the dimensionality is $d = 50$. For the adaptive sampler, the permutation matrix $Q$ is adaptively updated in every $50$ steps. For the exploratory projection sampler, the exploratory chain is run at a high temperature with $\beta_e = 0.1$ for $100,000$ steps. For each set of experiments, the samplers are given the same set of realizations of $(J_{i,j})$. The samplers are simulated for a total of $100,000$ steps.
		
		\item BC model on the line: the target inverse temperature is $\beta = 3$, and the dimensionality is $d = 50$. For the adaptive sampler, the permutation matrix $Q$ is adaptively updated in every $50$ steps. For the exploratory projection sampler, the exploratory chain is ran at a high temperature with $\beta_e = 0.1$ for $200,000$ steps. The samplers are simulated for a total of $200,000$ steps.
\end{itemize}}

{\color{black}\textbf{Convergence metrics and diagnostics}.} We assess convergence using the following metrics and diagnostics:

\begin{itemize}
	\item Magnetization against time. The magnetization of a configuration $\mathbf{x} \in \mathcal{X}$ is defined to be the average of the $d$ coordinates, that is,
	\begin{align*}
		\dfrac{1}{d} \sum_{i=1}^d x(i).
	\end{align*}
	
	\item \textcolor{black}{Fitted density curve of magnetization. Based on the histograms of magnetization, we fit and plot the density curves. For the Ising and BC model, we expect the fitted density curves to concentrate on the known modes in these models. }
	
	\textcolor{black}{\item Sample mean and 95\% confidence interval of magnetization. For the Ising and BC model, for good mixing of the samplers we expect the sample mean and confidence interval of magnetization to be concentrated around $0$ in low temperature.}
	
	\item Hamiltonian value $H$ against time. \textcolor{black}{As the temperature in the experiments is low, we expect the Hamiltonian value to decrease over time, reflecting that the samplers can be understood as ``noisy descent" type algorithms.}
	
	\textcolor{black}{\item Average jump distance of magnetization. We expect the projection samplers to have larger jump distance of magnetization compared with the standard MH. Note that the average jump distance of magnetization is defined to be
		\begin{align*}
			\dfrac{1}{T-1}\sum_{t=1}^{T-1} |m_{t+1} - m_t|,
		\end{align*}
		where $m_t$ is the magnetization at time $t$ and $T$ is the number of MCMC steps.}
\end{itemize}

Three set of experiments are conducted \textcolor{black}{for each model}. In each set, the standard MH serves as the baseline and is compared with the three projection samplers.

\subsection{List of commonly used notations}

 The following Table \ref{Table:summary} summarizes some frequently used notations throughout the paper.
\begin{table}[H]
	\begin{tabular}{cc}
		\toprule
		\textbf{Symbol} & \textbf{Meaning} \\
		\midrule
		$\mathcal{X}$ & Finite state space \\
		$\mathcal{L} = \mathcal{L}(\mathcal{X})$ & Set of transition matrices on $\mathcal{X}$ \\
		$\mathcal{P}(\mathcal{X})$ & Set of probability masses with full support on $\mathcal{X}$ \\
		$\llbracket m,n \rrbracket$ & $\{m, m+1, \dots, n\}$ for $m \leq n$ and $m,n \in \mathbb{Z}$\\
		$\llbracket n \rrbracket$ & $\llbracket 1,n \rrbracket$ for $n \in \mathbb{N}$\\
		$\ell^2(\pi)$ & $\pi$-weighted $\ell^2$ Hilbert space with inner product $\langle f,g \rangle_\pi = \sum_x f(x)g(x)\pi(x)$ \\
		$\ell^2_0(\pi)$ & $\{f \in \ell^2(\pi);~ \pi(f) = 0\}$ \\
		$\mathcal{S}(\pi)$ & Set of $\pi$-stationary transition matrices on $\mathcal{X}$ \\
		$\mathcal{L}(\pi)$ & Set of $\pi$-reversible transition matrices on $\mathcal{X}$ \\
		$P^*$ & $\ell^2(\pi)$-adjoint of $P$ \\
		$\Pi$ & Transition matrix where each row equals to $\pi \in \mathcal{P}(\mathcal{X})$ \\
		$\mathcal{I}(\pi)$ & Set of isometric involution matrices with respect to $\pi$ on $\mathcal{X}$ \\
		%$Q \in \mathcal{I}(\pi)$ & $Q^2 = I$, $Q^* = Q$, $Q : \ell^2(\pi) \to \ell^2(\pi)$ \\
		$\mathcal{L}(\pi, Q)$ & Set of $(\pi, Q)$-self-adjoint transition matrices: $L^* = QLQ$ for $L \in \mathcal{L}(\pi,Q)$ \\
		$\mathbf{P}$ & Set of permutations on $\mathcal{X}$ \\
		$\Psi(\pi)$ & Set of involution and equi-probability permutations with respect to $\pi$ \\
		$Q_\psi$ & Permutation matrix induced by $\psi \in \mathbf{P}$: $Q_\psi(x,y) = \delta_{y = \psi(x)}$ for all $x,y \in \mathcal{X}$ \\
		$D^\pi_{KL}(P \| L)$ & $\pi$-weighted KL divergence from $L$ to $P$ \\
		${}^Q D_{KL}(P \| L)$ & $Q$-left-deformed KL divergence from $L$ to $P$: $D_{KL}^\pi(QP \| QL)$ \\
		$D_{KL}^Q(P \| L)$ & $Q$-right-deformed KL divergence from $L$ to $P$: $D_{KL}^\pi(PQ \| LQ)$ \\
		$\overline{P}(Q)$ & $(\pi, Q)$-self-adjoint projection of $P$: $\frac{1}{2}(P + QP^*Q)$ \\
		$(R_l)_{l \in \mathbb{N}}$ & Sequence of alternating projections \\
		$R_\infty$ & Limit of alternating projections \\
		$P_\beta = P_\beta(N,H)$ & {\footnotesize Metropolis-Hastings with proposal $N$ and target Hamiltonian $H$ at inverse temperature $\beta$} \\
		$h(P_\beta)$ & Critical height associated with $P_\beta$ \\
		\bottomrule
	\end{tabular}
	\centering
	\caption{Summary of notations}\label{Table:summary}
\end{table}

\bibliographystyle{abbrvnat}
\bibliography{ref}

\end{document}